\documentclass[trsc]{informs4}
\RequirePackage{tgtermes}
\RequirePackage{newtxtext}
\RequirePackage{newtxmath}
\RequirePackage{bm}
\RequirePackage{endnotes}

\OneAndAHalfSpacedXII 

\usepackage{algorithm}
\usepackage{algpseudocode}
\usepackage{tikz}
\usepackage{xcolor}
\usepackage[normalem]{ulem}

\usepackage{natbib}
 \bibpunct[, ]{(}{)}{,}{a}{}{,}%
 \def\bibfont{\small}%
 \def\BIBand{and}%

\EquationsNumberedThrough

\TheoremsNumberedThrough 
\ECRepeatTheorems %

\usepackage{mathtools}
\usepackage{alphalph,etoolbox}

\usepackage{booktabs}
\usepackage{multirow}
\usepackage{tabularx}
\usepackage{siunitx}

\usepackage{algorithm}
\usepackage{algpseudocode}

\usepackage{caption}
\usepackage{subcaption}
\usepackage{placeins}
\usepackage{tikz}
\usetikzlibrary{arrows.meta,positioning,calc,shapes,fit}
\usepackage{pgfplots}
\pgfplotsset{compat=1.17}

\usepackage{natbib}
\bibpunct[, ]{(}{)}{,}{a}{}{,}%
\def\bibfont{\small}%
\def\BIBand{and}%
\usepackage[
  hidelinks,
  pdftitle={Multi-Stage Stochastic Optimization and Reinforcement Learning Approaches for Dynamic Inspection of Infrastructure Systems},
  pdfauthor={Juan-Alberto Estrada-Garcia and Siqian Shen}
]{hyperref}

\begin{document}


\RUNAUTHOR{Estrada-Garcia and Shen}

\RUNTITLE{Dynamic Inspection of Infrastructure Systems}

\TITLE{Multi-Stage Stochastic Optimization and Reinforcement Learning Approaches for Dynamic Inspection of Infrastructure Systems}

\ARTICLEAUTHORS{
\AUTHOR{Juan-Alberto Estrada-Garcia}
\AFF{Department of Industrial and Operations Engineering, University of Michigan, Ann Arbor, USA, \EMAIL{juanest@umich.edu}}

\AUTHOR{Siqian Shen}
\AFF{Department of Industrial and Operations Engineering, University of Michigan, Ann Arbor, USA, \EMAIL{siqian@umich.edu}}
} 

\ABSTRACT{
We study a dynamic inspection problem for infrastructure systems in which we route and schedule multiple vehicles to monitor components that are subject to heterogeneous stochastic failures. Inspection decisions determine which components the vehicles serve first and thus affect how failures propagate to neighboring components. We formulate the problem as a multi-stage stochastic mixed-integer program that jointly optimizes routing and scheduling under exogenous and decision-dependent uncertainty. We develop a stochastic dual dynamic integer programming (SDDiP) algorithm based on a stage decomposition of the multi-stage model and show that it converges finitely to the optimal value. In parallel, we propose a reinforcement learning (RL) approach that learns assignment policies from simulated system dynamics. On instances built from the road and storm-sewer networks of Sioux Falls, SDDiP proves optimality where the extended formulation of the same model leaves large gaps. On larger instances, RL with a routable-assignment certificate often finds cheaper policies than SDDiP under a short computational budget, whereas SDDiP finds cheaper policies and bounds on their quality with more time. Moreover, modeling failure propagation and decision-dependent transition probabilities reduces the realized cost when these mechanisms are present.
These comparisons guide the choice between optimization and policy learning for inspection planning under failure risk.
}


\maketitle
\section{Introduction}
\label{sec:intro}

Inspection and maintenance are central to the safety and reliability of large-scale infrastructure systems. Operators of geographically dispersed components, such as power networks, pipelines, bridges, and transportation corridors, must allocate limited inspection resources across locations with different hazards. In practice, these decisions are often sequential. On each shift, an operator selects a subset of components to visit and then determines operational routes that respect shift duration, depot locations, and vehicle capabilities.
A defining feature of many dynamic inspection problems is that inspection actions influence future failures. Timely inspections can reduce component failure risk by detecting deterioration conditions and triggering mitigation. Such feedback, however, creates the endogenous uncertainty of component failures, in which the distribution of future failures depends on the history of inspection decisions~\citep{goel2006class}.
These endogenous dynamics are particularly salient for infrastructure systems exposed to time-varying environmental stressors or cascading mechanisms, where inspection pays off through future risk reduction. In this paper, our base model treats the failure uncertainty as exogenous. We then extend it with two mechanisms of decision dependence: (i) failures that propagate to neighboring components until a vehicle serves the failed component, and (ii) transition probabilities that depend on intervention decisions.

A motivating application arises in electric power systems. Aging equipment, vegetation encroachment, and extreme weather can cause failures in transmission and distribution components. Such failures can propagate and lead to larger-scale service disruptions.
Similar planning challenges arise in civil infrastructure, where operators deploy ground crews and unmanned aerial systems over consecutive days to monitor bridges, railways, and pipelines under daily operational limits and evolving risk. An urban application, which we demonstrate later in our numerical studies, is the monitoring of storm-water drainage systems. During storms, a blocked inlet diverts water and damages surrounding infrastructure. Removing debris and clearing the inlet can prevent this damage and its propagation.

The dynamic inspection problem lies at the intersection of dynamic vehicle routing~\citep{Pillac2013AProblems}, scheduling~\citep{Memarzadeh2016IntegratedSystems},
and stochastic dynamic optimization.
Two-stage optimization models and static routing policies cannot capture how monitoring adapts sequentially when the inspection decisions affect the subsequent risk of failures. Multi-stage stochastic mixed-integer programming (SMIP) provides a natural modeling framework. However, exact scenario-tree formulations become computationally intractable as the number of stages and scenarios grows~\citep{shapiro2021lectures}.
To address this challenge, we propose a stage decomposition of the multi-stage model. In every stage, which represents a working shift, each vehicle selects one route, which fixes both its set of tasks and its schedule. We further design a reinforcement learning (RL) framework that learns assignment policies through interaction with simulated trajectories that capture endogenous failures and system dynamics.

\subsection{Contributions}
This paper makes three main contributions. First, we propose an exact route-based stage decomposition of a multi-stage inspection model that jointly optimizes the assignment, routing, and scheduling of the vehicles under failure risk. The state tracks completed inspections, realized failures, and vehicle depot locations. We extend the model with cross-task failure propagation (Section~\ref{sec:propagation_extension}) and with decision-dependent transition probabilities (Section~\ref{sec:decision_dependent_prob}).
Second, we propose an SDDiP-type algorithm that converges to an optimal policy in finitely many iterations (Section~\ref{sec:snd}). It solves the route-based formulation with cutting planes that are generated to approximate the cost-to-go functions. The algorithm differs from standard SDDiP in three respects: (i) integer optimality cuts restricted to the state coordinates that affect the value, (ii) perspective cuts for decision-dependent transitions, and (iii) valid inequalities that penalize states without a feasible completion. We prove finite convergence of the algorithm.
Third, we formulate the assignment decisions as a finite-horizon Markov decision process (MDP) and develop a hierarchical actor-critic to learn an assignment policy (Section~\ref{sec:rl}). We compare this policy with SDDiP using a common objective and simulator.

\subsection{Structure of the Paper}
The remainder of the paper is organized as follows. 
Section~\ref{sec:lit} reviews relevant research and positions our work in the recent literature. 
Section~\ref{sec:formulation} presents the multi-stage SMIP model and its two extensions. Section~\ref{sec:snd} develops the SDDiP-type exact algorithm, and Section~\ref{sec:rl} describes an alternative RL algorithm for finding dynamic inspection policies. Section~\ref{sec:experiments} reports numerical studies that compare different models and solution approaches on instances with diverse topologies and failure configurations. Section~\ref{sec:conclu} concludes the paper and discusses future research.

\section{Related Literature and Comparisons}
\label{sec:lit}

Inspection routing belongs to the broader family of vehicle routing problems (VRPs) that incorporate heterogeneous constraints and operational features beyond the canonical VRP. In inspection, shift durations and vehicle capabilities limit each route, and the failure risk of each component sets its service priority.
Infrastructure operators coordinate inspection and maintenance with life-cycle management and operational resilience. \citet{phan2015multistage} develop a multi-stage optimization framework that groups geographically distributed components, determines inspection intervals, and schedules inspections under workforce and budget constraints. Our base model in Section~\ref{sec:exactSDP} differs from theirs in that (i) we couple inspection with multi-vehicle routing within each stage, and (ii) we require the vehicles to serve every task by the end of the horizon, including a task that has already failed.

For drone inspection, \citet{chung2020placement} study inspection placement and routing for offshore wind farms. \citet{fu2021real} study unmanned aerial vehicle (UAV) routing for post-disaster inspection of power distribution networks. \citet{LiuDroneInspection} study the joint problem of drone station location and routing for infrastructure inspection. Their model incorporates heterogeneous drones and solves a static location-routing formulation with adaptive large neighborhood search, but in a deterministic setting without the dynamic evolution of infrastructure conditions. In contrast to these static models, the decisions of our model adapt to the observed failure status over a multi-stage horizon.

From an operational perspective, one can interpret dynamic inspection planning as a form of dynamic vehicle routing in which the operator revises inspection assignments and routes as the system state evolves. Surveys of dynamic VRP emphasize the offline--online structure that arises when the decision maker must adapt decisions to newly revealed information~\citep{rios2021recent}. In many canonical dynamic VRP settings, uncertainty is exogenous (decision-independent), for example, random customer requests and service times~\citep{Errico2018,Ulmer2019}. In contrast, inspection settings often involve endogenous uncertainty, in which inspection and maintenance actions influence the evolution of system condition and failure risk. Our base model treats the failure uncertainty as exogenous, which we call decision-independent uncertainty (DIU), but its inspection decisions determine the realized failures and the operational state. 
Following \citet{goel2006class}, decision-dependent uncertainty (DDU) is of Type~1 when decisions alter probability distributions and of Type~2 when they alter only the time at which the decision maker observes uncertain outcomes. The decision-dependent transition probabilities of Section~\ref{sec:decision_dependent_prob} are of Type~1.

For DDU,~\citet{shahverdi2023scheduling} schedule inspection and restoration actions under decision-dependent timing, in which inspection reveals a component's status. Because the authors assume negligible travel times along open network paths, they do not optimize routing. \citet{EstradaInfrastructureMonitoring} consider inspection routing and scheduling under decision-dependent failures in a two-stage setting, in which the inspection decisions filter the failure scenarios of individual components. Our extension in Section~\ref{sec:propagation_extension} models a different mechanism. A component that has failed and remains unserved shortens the failure times of the components that depend on it until a vehicle serves it. The service decisions therefore change the failure times of other components, and these effects accumulate over the stages of the horizon.

 To model DDU, \citet{morton2026mdp} let binary decisions select transition matrices from a finite set in a policy-graph model, while \citet{kang2025distributionally} study decisions that change ambiguity sets in multi-stage distributionally robust optimization. \citet{kang2025bi} develop exact methods for two-stage problems with binary decisions that parameterize the recourse. In Section~\ref{sec:decision_dependent_prob}, we consider another extension by incorporating decision-dependent transitions into multi-stage inspection routing. There, the intervention intensity of an assignment selects a transition probability vector from a finite set, one vector per intervention level. We do not model ambiguity sets or decision-dependent information revelation.

We now turn to the literature on multi-stage SMIP and its use in dynamic VRPs. One can formulate multi-stage stochastic optimization models on scenario trees, optimizing the decisions of all nodes jointly. The resulting models quickly become computationally intractable as the number of stages grows~\citep{shapiro2021lectures}. In multi-stage settings with binary state variables, stochastic dual dynamic integer programming (SDDiP) constructs nested lower approximations of cost-to-go functions to decompose the problem by stages~\citep{Zou2019StochasticProgramming}. For routing, arc-based stage formulations suffer from weak linear programming (LP) relaxations and hence weak dual-based cuts. Our algorithm keeps the binary state and the cut families of SDDiP, but it solves a route-based problem at every node of the scenario tree, as in set-partitioning formulations of vehicle routing~\citep{DesrosiersLubbecke2005}. To the best of our knowledge, our work is among the first applications of SDDiP to dynamic inspection routing problems with endogenous infrastructure dynamics and uncertainty.

In stochastic dynamic routing, RL is most promising when policies exploit structured state representations~\citep{hildebrandt2023opportunities}. Hierarchical RL separates higher-level choices from lower-level execution~\citep{vezhnevets2017feudal}. \citet{morton2026mdp} further formalize the theoretical link between multi-stage stochastic programming and Markov decision processes, which provides a foundation for comparisons between optimization-based and learning-based approaches. We use a hierarchical actor-critic framework to obtain a task assignment policy that we compare with the policies SDDiP produces.

\begin{table}[htbp]
\centering
\footnotesize
\setlength{\tabcolsep}{3pt}
\caption{Comparison of the most relevant literature.}
\label{table:LitRevInspection}
\begin{tabular}{llllll}
\toprule
Work & Application & Uncertainty & DDU type & Solution & Horizon \\
\midrule
\citet{Ulmer2019} & VRP & DIU & -- & ADP & Multi-stage \\
\citet{Errico2018} & VRP & DIU & -- & SO & Multi-stage \\
\citet{nazari2018reinforcement} & VRP & DIU & -- & RL & Single-stage \\
\citet{LiuDroneInspection} & Inspection routing & Deterministic & -- & ALNS & Single-stage \\
\citet{EstradaInfrastructureMonitoring} & Inspection routing & DDU & Filtering & SD+CG & Two-stage \\
\citet{shahverdi2023scheduling} & Inspection scheduling & DDU & Timing & SO & Multi-stage \\
\citet{zou2018multistage} & Power systems & DIU & -- & SDDiP & Multi-stage \\
\citet{kang2025bi} & General SMIP & DDU & Recourse & SO & Two-stage \\
\citet{morton2026mdp} & General SMIP & DDU & Distribution & SO/MDP & Multi-stage \\
\textbf{Our paper} & \textbf{Inspection routing} & \textbf{DDU} & \begin{tabular}[t]{@{}l@{}}\textbf{Distribution,}\\ \textbf{Failure times}\end{tabular} & \textbf{SDDiP/RL} & \textbf{Multi-stage} \\
\bottomrule
\end{tabular}
\begin{minipage}{0.98\textwidth}
\footnotesize
\vspace{1pt}
ADP: approximate dynamic programming. SO: stochastic optimization. ALNS: adaptive large
neighborhood search. SD: scenario decomposition. CG: column generation. RL: reinforcement
learning. DIU/DDU: decision-independent/decision-dependent uncertainty. DDU type: decisions
filter the scenarios (Filtering), alter observation timing (Timing), alter
probability distributions (Distribution), change failure times (Failure times), or parameterize the recourse (Recourse).
\end{minipage}
\end{table}

Table~\ref{table:LitRevInspection} compares the models and solution methods in the related literature. Our formulation combines multi-stage inspection routing with failure propagation and intervention-dependent transition probabilities. We address the computational challenges of this combined setting by proposing a route-based decomposition and an RL policy.

\section{Multi-stage Stochastic Dynamic Inspection Model and Extensions}
\label{sec:formulation}
We describe the base model and the extended formulation of the multi-stage SMIP in Section~\ref{sec:exactSDP}. We extend it with cross-task failure propagation in Section~\ref{sec:propagation_extension} (Extension~1) and decision-dependent transition probabilities in Section~\ref{sec:decision_dependent_prob} (Extension~2). 

\subsection{Base Model and Extended Formulation}
\label{sec:exactSDP}

Consider a finite-horizon stochastic dynamic inspection problem defined on a directed graph $G=(N,A)$, whose vertex set $N$ we partition into a set of inspection tasks $\mathcal{T}\subset N$ and a set of depots $\mathcal{D}\subset N$. We operate a fleet of vehicles indexed by $k\in K$ over a planning horizon of $T$ stages $t\in\{1,\dots,T\}$. The arc set $A \subseteq (\mathcal{D}\cup\mathcal{T}) \times (\mathcal{D}\cup\mathcal{T})$ contains all admissible depot--task, task--task, and task--depot arcs. We assume that no vehicle stays idle, that is, each vehicle serves at least one inspection task at every stage. Each vehicle $k$ starts the horizon at an initial depot and may end each stage at any depot of a set $\mathcal{D}_k\subseteq\mathcal{D}$ of admissible depots. A vehicle starts each stage at the depot at which it ended the previous stage. A vehicle whose set $\mathcal{D}_k$ contains only its initial depot is home-based. A district is a depot together with the vehicles based there and the tasks they may serve. Traversal of arc $(i,j)\in A$ by vehicle $k$ requires travel time $\delta_{ijk}\ge 0$, and serving task $i\in\mathcal{T}$ with vehicle $k$ requires service time $\kappa_{ik}\ge 0$. We assume that the sum of travel and service-time increments around every directed cycle composed only of inspection tasks is strictly positive, as holds in particular when service times are positive.
In each stage, vehicle $k$ operates for a fixed duration $\gamma_k>0$, which represents one working shift, and becomes available at time $\theta_{kt}$ in stage $t$. We encode the end-of-stage time as $E_t$ and assume that $\theta_{kt}+\gamma_k \le E_t$ for all $k\in K$ and $t\in\{1,\dots,T\}$. Moreover, vehicle $k$ can serve at most $B_{kt}$ inspection tasks at stage $t$.

Each component associated with inspection task $i\in\mathcal{T}$ is subject to stochastic failure. Serving the task means inspecting the component and, if needed, performing an intervention. Service before the failure time prevents the failure, and a failure before service incurs the penalty $\alpha_i\ge0$ once. Within the horizon, completion and failure are permanent, and a failed task still must be served. Therefore, the vehicles serve every task exactly once within the horizon in every scenario. Failures do not block access or change travel times. At the start of each stage, the operator observes the failure status of all tasks. In storm-water systems, for example, flooding reports and service requests reveal failures.

We represent failure time uncertainty of tasks with a scenario tree. Let $\mathcal{N}_t$ denote the set of nodes at stage $t$, with root node $n_1\in\mathcal{N}_1$ and $|\mathcal{N}_1|=1$. The overall set of scenario-tree nodes is $\mathcal{N} = \cup_{t=1,\ldots,T} \mathcal{N}_t$. For each nonroot node $n$, let $n^{-}$ denote its unique parent. For each node $n$, let $\mathcal{C}(n)$ denote its set of child nodes and $p_{nm}$ the conditional probability of moving from $n$ to child $m\in\mathcal{C}(n)$. Let $t(n) \in \{1,\dots,T\}$ denote the stage associated with node $n$. We denote by $\Pi(n)$ the sequence of nodes along the unique path from the root to node $n$. The probability of node $n$ is $\mathbb{P}(n):=\prod_{(j,m)\text{ on }\Pi(n)}p_{jm}$, with $\mathbb{P}(n_1)=1$. We assume that $p_{nm}>0$ for every $n\in\mathcal{N}$ and $m\in\mathcal{C}(n)$. Thus, every root-to-leaf path of the tree has positive probability under every policy. We refer to this as the positive-support assumption. For node $n \in \mathcal{N}_t$, the parameter $\varphi_i^n\ge 0$ denotes the failure time of task $i \in \mathcal{T}$ if it is not served. The parameter $\bar{f}^{n}_{i} := \mathbb{I}\{\varphi^{n}_{i} \leq E_{t(n)}\}$ indicates whether the failure of task $i$ will occur by the end of stage $t(n)$. Figure~\ref{fig:ScenarioTree} illustrates the scenario tree and its notation.

\begin{figure}[htbp!]
    \centering
    \includegraphics[width=0.55\textwidth]{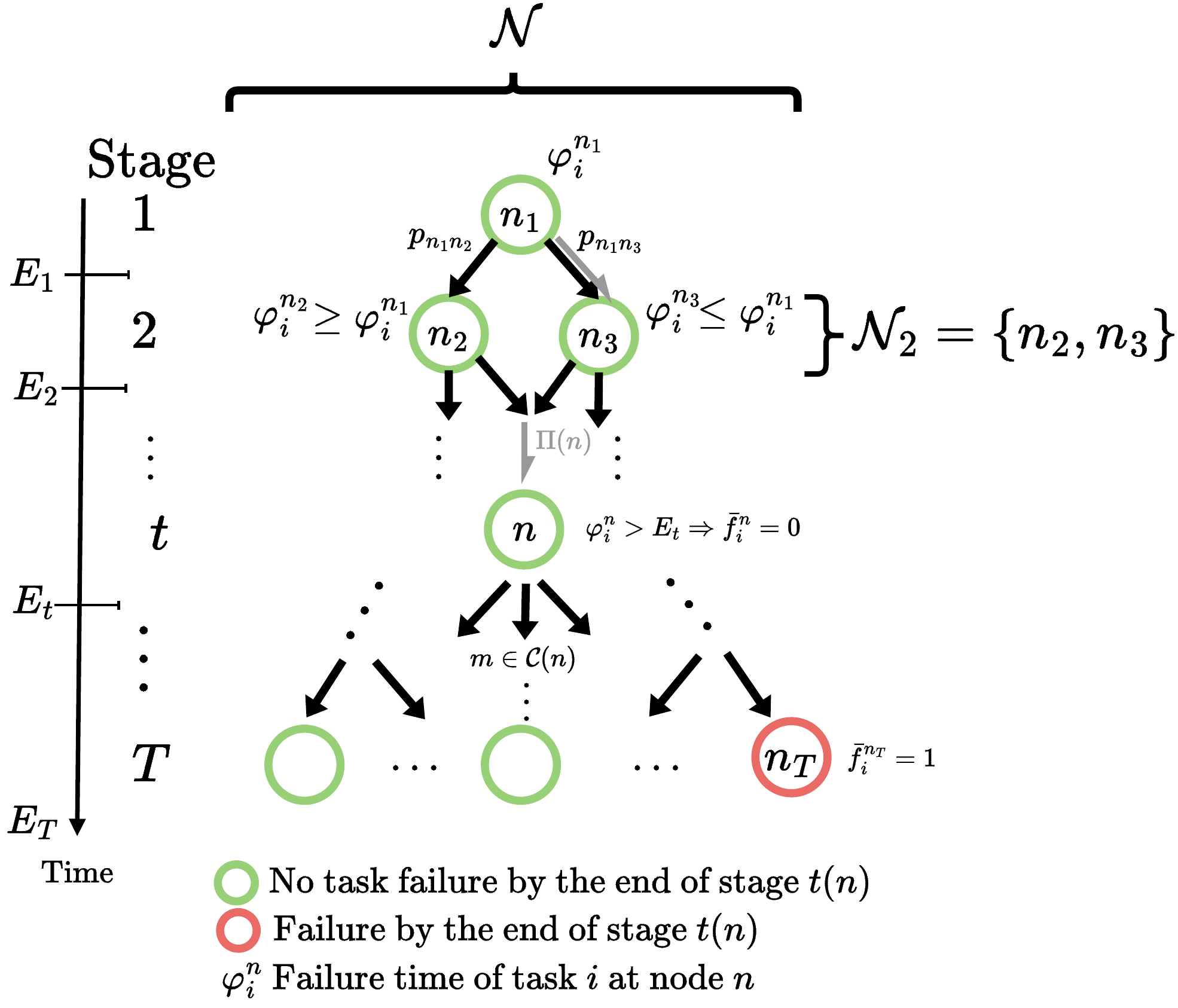}
    \caption{Scenario tree representation of the failure timing of inspection tasks as a function of the stage. The path of gray arrows represents $\Pi(n)$. We make the decisions of stage $t$ at the nodes of $\mathcal{N}_t$, after observing the information of stage $t$.}
    \label{fig:ScenarioTree}
\end{figure}

At each node $n\in\mathcal{N}_t$, decision variables describe assignment, routing, and timing, while state variables record inspection and failure outcomes. The binary variable $x_{ik}^n\in\{0,1\}$ equals one if vehicle $k$ serves task $i$ at node $n$, and zero otherwise. The binary variable $y_{ijk}^n\in\{0,1\}$ equals one if vehicle $k$ traverses arc $(i,j)$ at node $n$, and zero otherwise. The binary variables $b_{dk}^n\in\{0,1\}$ and $e_{dk}^n\in\{0,1\}$ indicate, respectively, the depot from which vehicle $k$ departs and the depot at which it completes its route within stage $t(n)$. The variable $a_{ik}^n\ge 0$ denotes the completion time of task $i$ by vehicle $k$ at node $n$. For each task $i \in \mathcal{T}$ and node $n \in \mathcal{N}$, we also define an auxiliary binary variable $v_i^n\in\{0,1\}$ that equals one if a vehicle has completed task $i$ by the end of stage $t(n)$ at node $n$. We call task $i$ open at node $n$ if $v_i^n=0$. The variable $r_{ik}^n\in\{0,1\}$ equals one if vehicle $k$ completes task $i$ at node $n$ before its failure time. The variable $q_i^n\in\{0,1\}$ equals one if task $i$ is in a failed state at node $n$. At node $n$, the vectors $\mathbf{v}^n$, $\mathbf{q}^n$, and $\mathbf{e}^n$ collect the task completion, failure, and vehicle end-depot indicators, respectively. We write $n_1^{-}$ for the information that enters the root node $n_1$, so that $\mathbf{v}^{n_1^{-}}$, $\mathbf{q}^{n_1^{-}}$, and $\mathbf{e}^{n_1^{-}}$ are the initial state values.

Next, we describe and group the constraints at each node into three blocks. The assignment block $\mathcal{A}^n$ restricts the tasks assigned to each vehicle. The routing block $\mathcal{O}^n_k$ describes the route and the schedule of vehicle $k$ in the stage. The state block $\mathcal{S}^n$ updates the completion and failure indicators.

\paragraph{Assignment block.} For node $n$, $\mathcal{A}^n$ is the set of assignments $\mathbf{x}^n$ satisfying
\begin{subequations}\label{ABLOCK}
\begin{align}
& \sum_{i\in\mathcal{T}} x_{ik}^n \le B_{k t(n)},
\quad \forall k\in K,
\label{EXT:budget} \\
& \sum_{i\in\mathcal{T}} x_{ik}^n \ge 1,
\quad \forall k\in K,
\label{EXT:active_vehicle} \\
& x_{ik}^n\in\{0,1\},
\quad \forall i\in\mathcal{T},\ k\in K.
\label{EXT:bin_x}
\end{align}
\end{subequations}
Constraint~\eqref{EXT:budget} limits the number of tasks assigned to each vehicle at each node, and constraint~\eqref{EXT:active_vehicle} rules out idle vehicles.

\paragraph{Routing block.} For routing, we replace every depot $d\in\mathcal{D}$ by a source copy $d^+$ and a sink copy $d^-$, such that a vehicle can start and end at the same depot without creating a time cycle. The arc set $\tilde{A}$ of the copied graph contains the depot-to-task arcs from source copies, the task-to-task arcs, and the task-to-depot arcs into sink copies. For node $n$, vehicle $k$, its assignment $\mathbf{x}^n_k$, and the depot $\mathbf{e}^{n^{-}}_k$ at which it ended the previous stage, $\mathcal{O}^n_k(\mathbf{x}^n_k,\mathbf{e}^{n^{-}}_k)$ is the set of $(\mathbf{y}^n_{k},\mathbf{b}^n_{k},\mathbf{e}^n_{k},\mathbf{a}^n_{k},\mathbf{r}^n_k)$ satisfying
\begin{subequations}\label{OP}
\begin{align}
& \sum_{d\in\mathcal{D}} b_{dk}^n = 1,
\label{OP:start_depot} \\
& \sum_{d\in\mathcal{D}} e_{dk}^n = 1,
\label{OP:end_depot} \\
& \sum_{j:(i,j)\in \tilde{A}} y_{ijk}^n = x_{ik}^n,
\quad \forall i\in\mathcal{T},
\label{OP:task_out} \\
& \sum_{j:(j,i)\in \tilde{A}} y_{jik}^n = x_{ik}^n,
\quad \forall i\in\mathcal{T},
\label{OP:task_in} \\
& \sum_{j:(d^+,j)\in \tilde{A}} y_{d^+jk}^n = b_{dk}^n,
\quad \forall d\in\mathcal{D},
\label{OP:depot_out} \\
& \sum_{i:(i,d^-)\in \tilde{A}} y_{id^-k}^n = e_{dk}^n,
\quad \forall d\in\mathcal{D},
\label{OP:depot_in} \\
& a_{jk}^n \ge \theta_{k t(n)}+\delta_{d^+jk}+\kappa_{jk}-M(1-y_{d^+jk}^n),
\quad \forall (d^+,j)\in\tilde{A},
\label{OP:depot_time_lb} \\
& a_{jk}^n \ge a_{ik}^n + \delta_{ijk} + \kappa_{jk} - M(1-y_{ijk}^n),
\quad \forall (i,j)\in \tilde{A},\ i,j\in\mathcal{T},
\label{OP:time_arc} \\
& a_{ik}^n \le \theta_{k t(n)} + \gamma_k + M(1-x_{ik}^n),
\quad \forall i\in\mathcal{T},
\label{OP:shift_task} \\
& a_{ik}^n+\delta_{id^-k} \le \theta_{k t(n)}+\gamma_k+M(1-y_{id^-k}^n),
\quad \forall (i,d^-)\in\tilde{A},
\label{OP:shift_depot} \\
& b_{dk}^n = e_{dk}^{n^{-}},
\quad \forall d\in\mathcal{D},
\label{EXT:carry} \\
& e_{dk}^{n}=0,
\quad \forall d\in\mathcal{D}\setminus\mathcal{D}_k,
\label{EXT:home} \\
& r_{ik}^{n} \le x_{ik}^{n},
\quad \forall i\in\mathcal{T},
\label{EXT:r_assign_agg} \\
& a_{ik}^n \le \varphi_i^n + M(1-r_{ik}^n),
\quad \forall i\in\mathcal{T},
\label{EXT:r_time_agg}\\
& y_{ijk}^n\in\{0,1\},
\quad \forall (i,j)\in \tilde{A},
\label{EXT:bin_y} \\
& b_{dk}^n,e_{dk}^n\in\{0,1\},
\quad \forall d\in\mathcal{D},
\label{EXT:bin_depot} \\
& r_{ik}^n\in\{0,1\},
\quad \forall i\in\mathcal{T},
\label{EXT:bin_r} \\
& a_{ik}^n \ge 0,
\quad \forall i\in\mathcal{T}.
\label{EXT:cont}
\end{align}
\end{subequations}
Constraints~\eqref{OP:start_depot}--\eqref{OP:depot_in} select one start depot and one end depot and impose flow conservation through the assigned tasks. Constraints~\eqref{OP:depot_time_lb}--\eqref{OP:time_arc} propagate completion times along the route. Constraints~\eqref{OP:shift_task}--\eqref{OP:shift_depot} keep the completion times and the return to the depot within the interval $[\theta_{kt(n)},\theta_{kt(n)}+\gamma_k]$. Constraint~\eqref{EXT:carry} starts each vehicle at the depot where it ended at node $n^{-}$. Constraint~\eqref{EXT:home} restricts its end depot to $\mathcal{D}_k$. Constraints~\eqref{EXT:r_assign_agg}--\eqref{EXT:r_time_agg} credit an on-time completion $r_{ik}^n=1$ only to the assigned vehicle and only if $a_{ik}^n\le\varphi_i^n$. Constraints~\eqref{EXT:bin_y}--\eqref{EXT:bin_r} make $\mathbf{y}_k^n$, $\mathbf{b}_k^n$, $\mathbf{e}_k^n$, and $\mathbf{r}_k^n$ binary. Constraint~\eqref{EXT:cont} makes the completion times nonnegative.

We use valid big-$M$ constants based on \citet{desrochers1991improvements}: $M=\theta_{kt(n)}+\delta_{d^+jk}+\kappa_{jk}$ in constraint~\eqref{OP:depot_time_lb}, $M=\theta_{kt(n)}+\gamma_k+\delta_{ijk}+\kappa_{jk}$ in constraint~\eqref{OP:time_arc}, and $M=\delta_{id^-k}$ in constraint~\eqref{OP:shift_depot}. Because constraint~\eqref{OP:shift_task} holds with $M=0$, we impose it as the upper bound $\theta_{kt(n)}+\gamma_k$ on every completion time. In constraint~\eqref{EXT:r_time_agg}, $M=\max\{0,\theta_{kt(n)}+\gamma_k-\varphi^n_i\}$, the largest amount by which a completion time in the stage can exceed the failure time.

The timing constraints~\eqref{OP:time_arc} eliminate subtours among tasks. Summing these constraints around a cycle of selected arcs cancels the completion times and requires the sum of the travel and service times around the cycle to be nonpositive. This contradicts our assumption that the travel and service times around every cycle of tasks sum to a positive value.

\paragraph{State block.} For node $n$, $\mathcal{S}^n(\mathbf{v}^{n^{-}},\mathbf{q}^{n^{-}},\mathbf{x}^n,\mathbf{r}^n)$ is the set of $(\mathbf{v}^n,\mathbf{q}^n)$ satisfying
\begin{subequations}\label{SBLOCK}
\begin{align}
& \sum_{k\in K} x_{ik}^n \le 1 - v_i^{n^{-}},
\quad \forall i\in\mathcal{T},
\label{EXT:once} \\
& v_i^n = v_i^{n^{-}} + \sum_{k\in K} x_{ik}^n,
\quad \forall i\in\mathcal{T},
\label{EXT:v_update} \\
& r_{ik}^n \le 1-q_{i}^{n^{-}},
\quad \forall i\in\mathcal{T},\ k\in K,
\label{EXT:r_fail_agg} \\
& q_i^n \ge q_i^{n^{-}},
\quad \forall i\in\mathcal{T},
\label{EXT:q_monotone} \\
& q_i^n \ge \bar{f}_i^{n} - v_i^{n^{-}} - q_i^{n^{-}} - \sum_{k\in K} r_{ik}^n,
\quad \forall i\in\mathcal{T},
\label{EXT:q_fail_clean} \\
& q_i^n \le q_i^{n^{-}}+1-v_i^{n^{-}},\quad \forall i\in\mathcal{T},\label{EXT:q_no_postcompletion} \\
& q_i^n \le q_i^{n^{-}}+\bar{f}_i^n,\quad \forall i\in\mathcal{T},\label{EXT:q_upper_failure} \\
& q_i^n \le q_i^{n^{-}}+1-\sum_{k\in K}r_{ik}^n,\quad \forall i\in\mathcal{T},\label{EXT:q_upper_ontime} \\
& v_i^n,q_i^n\in\{0,1\},
\quad \forall i\in\mathcal{T}.
\label{EXT:bin_state}
\end{align}
\end{subequations}
Constraint~\eqref{EXT:once} prevents the re-inspection of a task completed at an ancestor node, and constraint~\eqref{EXT:v_update} updates the completion indicator. Constraint~\eqref{EXT:r_fail_agg} rules out an on-time completion of a task that has already failed at the parent node. Constraint~\eqref{EXT:q_monotone} keeps a failed task failed at all descendant nodes. Constraint~\eqref{EXT:q_fail_clean} triggers the failure of task $i$ at node $n$ when $\bar{f}_i^n=1$, unless a vehicle had completed the task or the task had failed at the parent node ($v_i^{n^{-}}+q_i^{n^{-}}\ge1$) or a vehicle completes it on time at node $n$. Conversely, when $\bar{f}_i^n=0$, when a vehicle had completed the task at the parent node, or when a vehicle completes it on time, constraints~\eqref{EXT:q_no_postcompletion}--\eqref{EXT:q_upper_ontime} yield $q_i^n\le q_i^{n^{-}}$. Constraint~\eqref{EXT:bin_state} makes the completion and failure indicators binary.

Let $\bar{e}_{dk}$ equal one if vehicle $k$ starts at depot $d$ before the root decision, and zero otherwise. We now formulate the dynamic inspection problem as the scenario-tree-indexed mixed-integer linear program~\eqref{EXT}, which we call the extended formulation of the base model:
\begin{subequations}\label{EXT}
\begin{align}
\min_{\substack{\mathbf{x},\mathbf{y},\mathbf{b},\mathbf{e},\\\mathbf{v},\mathbf{r},\mathbf{q},\mathbf{a}}} \quad
& \sum_{t=1}^{T} \sum_{n\in\mathcal{N}_t} \mathbb{P}(n)
 \Bigl[\sum_{k\in K}\sum_{(i,j)\in \tilde{A}} \delta_{ijk}\,y_{ijk}^n
\ +\sum_{i\in\mathcal{T}}\alpha_i\bigl(q_i^n-q_i^{n^{-}}\bigr)\Bigr]
\label{EXT:obj} \\
\text{s.t.}\quad
& \mathbf{x}^n\in\mathcal{A}^n,
\quad \forall n\in\mathcal{N},
\nonumber \\
& (\mathbf{y}^n_{k},\mathbf{b}^n_{k},\mathbf{e}^n_{k},\mathbf{a}^n_{k},\mathbf{r}^{n}_k)\in \mathcal{O}^n_k(\mathbf{x}^n_k,\mathbf{e}^{n^{-}}_k),
\quad \forall k \in K,\ n\in\mathcal{N},
\nonumber \\
& (\mathbf{v}^n,\mathbf{q}^n)\in\mathcal{S}^n(\mathbf{v}^{n^{-}},\mathbf{q}^{n^{-}},\mathbf{x}^n,\mathbf{r}^n),
\quad \forall n\in\mathcal{N}, \nonumber
\\
& v_{i}^{n_1^{-}} = 0,\quad \forall i \in \mathcal{T},
\label{EXT:init_v}\\
& q_{i}^{n_1^{-}} = 0,\quad \forall i \in \mathcal{T},
\label{EXT:init_q}\\
& e_{dk}^{n_1^{-}}=\bar{e}_{dk},\quad \forall d\in\mathcal{D},\ k\in K,
\label{EXT:init_e}\\
& v_i^n=1,
\quad \forall i\in\mathcal{T},\ n\in\mathcal{N}_T.
\label{EXT:terminal_coverage}
\end{align}
\end{subequations}
The objective~\eqref{EXT:obj} minimizes expected total cost over the scenario tree by accumulating expected travel time across all nodes and stages and the failure cost $\alpha_i$, which it charges once, when $q_i$ changes from zero to one. We impose the assignment block~\eqref{ABLOCK}, the routing block~\eqref{OP}, and the state block~\eqref{SBLOCK} at each node. Constraints~\eqref{EXT:init_v}--\eqref{EXT:init_q} initialize the task completion and failure indicators at the root. Constraint~\eqref{EXT:init_e} sets the initial depots, and constraint~\eqref{EXT:terminal_coverage} requires the vehicles to complete every task by the end of the horizon, including failed tasks.

In the base model, the schedule determines whether a task is served before its failure time. Those failure times are inputs to the model. In networked systems, tasks may interact such that serving one can also change how soon another may fail. In the next section, we capture this effect by allowing a failed task to accelerate the failure times of the tasks that depend on it.

\subsection{
Extension 1: Cross-Task Propagation}
\label{sec:propagation_extension}

We extend the base model~\eqref{EXT} to address cross-task propagation. For example, in a storm-water drainage system, a clogged inlet that is not cleared diverts its runoff to the inlets downstream of it and shortens their failure times until a vehicle clears it. We use a sparse directed graph $\mathcal{E}\subseteq\mathcal{T}\times\mathcal{T}$, the propagation graph, to identify which failures can propagate. For task $i$, the set $\operatorname{Pred}(i):=\{\ell:(\ell,i)\in\mathcal{E}\}$ contains the tasks whose failure can accelerate the failure of task $i$. Let $\varphi_i^{n,0}$ denote the failure time of the base model and $\bar{f}_i^{n,0}:=\mathbb{I}\{\varphi_i^{n,0}\le E_{t(n)}\}$. Task $i$ is a propagation source at node $n$ if its base failure time has passed and no vehicle has served it by the end of stage $t(n)$. We let $O_i^n$ indicate whether task $i$ is a propagation source at node $n$. Each failed and unserved predecessor $\ell\in\operatorname{Pred}(i)$ shortens the failure time of task $i$ by $\zeta_{\ell i}\ge0$. Let $\bar{\zeta}_i:=\sum_{\ell\in\operatorname{Pred}(i)}\zeta_{\ell i}$, and let $\mathcal{I}:=\{(n,i)\in\mathcal{N}\times\mathcal{T}:\varphi_i^{n,0}-\bar{\zeta}_i\le E_{t(n)}<\varphi_i^{n,0}\}$ be the pairs whose failure status depends on propagation. Extension 1 (Ext.\ 1) replaces the constant failure times $\varphi_i^n$ and failure indicators $\bar f_i^n$ of blocks~\eqref{OP} and~\eqref{SBLOCK} with variables, and its formulation is given by model~\eqref{EXT1}:
\begin{subequations}\label{EXT1}
\begin{align}
\min_{\substack{\mathbf{x},\mathbf{y},\mathbf{b},\mathbf{e},\mathbf{v},\mathbf{r},\mathbf{q},\\ \mathbf{a},\mathbf{O},\bm{\varphi},\bar{\mathbf{f}}}}\quad
& \sum_{t=1}^{T}\sum_{n\in\mathcal{N}_t}\mathbb{P}(n)\Bigl[\sum_{k\in K}\sum_{(i,j)\in\tilde{A}}\delta_{ijk}y^n_{ijk}+\sum_{i\in\mathcal{T}}\alpha_i\bigl(q^n_i-q^{n^{-}}_i\bigr)\Bigr]
\label{EXT1:obj}\\
\text{s.t.}\quad
& \eqref{ABLOCK}\text{--}\eqref{SBLOCK},\ \eqref{EXT:init_v}\text{--}\eqref{EXT:terminal_coverage},\nonumber\\
& O_i^n=\bar f_i^{n,0}\bigl(1-v_i^n\bigr),\quad \forall i\in\mathcal T,\ n\in\mathcal N,
\label{eq:propagation_trigger}\\
& \varphi_i^n=\varphi_i^{n,0}-\sum_{\ell\in\operatorname{Pred}(i)}\zeta_{\ell i}\,O_\ell^{n^{-}},\quad \forall i\in\mathcal{T},\ n\in\mathcal{N},
\label{eq:propagation_effective}\\
& \varphi_i^n\le E_{t(n)}+\bigl(\varphi_i^{n,0}-E_{t(n)}\bigr)\bigl(1-\bar{f}_i^n\bigr),\quad \forall (n,i)\in\mathcal{I},
\label{eq:fbar_link_ub}\\
& \varphi_i^n\ge E_{t(n)}+\eta-\bigl(E_{t(n)}+\eta-\varphi_i^{n,0}+\bar{\zeta}_i\bigr)\bar{f}_i^n,\quad \forall (n,i)\in\mathcal{I},
\label{eq:fbar_link_lb}\\
& \bar f_i^n\in\{0,1\},\quad \forall (n,i)\in\mathcal{I},\label{eq:fbar_domain1}\\ 
&\bar f_i^n=\bar f_i^{n,0},\quad \forall (n,i)\in(\mathcal N\times\mathcal T)\setminus\mathcal{I},
\label{eq:fbar_domain2}
\end{align}
\end{subequations}
where $O_i^{n_1^{-}}:=0$ and $\eta>0$ is smaller than every positive difference between an attainable value of $\varphi_i^n$ and $E_{t(n)}$. Constraint~\eqref{eq:propagation_trigger} marks task $i$ as a source at node $n$. Because the trigger uses the base failure indicator, only failures that occur without propagation act as sources. Thus, a failure caused by propagation does not propagate further. A source stops propagating as soon as a vehicle serves it, because $v_i^n=1$ then switches the trigger off. Constraint~\eqref{eq:propagation_effective} shortens the failure time of task $i$ by the reductions of its failed and unserved predecessors at the parent node. Constraint~\eqref{eq:fbar_link_ub} allows $\bar f_i^n=1$ only if $\varphi_i^n\le E_{t(n)}$. Conversely, constraint~\eqref{eq:fbar_link_lb} allows $\bar f_i^n=0$ only if $\varphi_i^n>E_{t(n)}$. Constraint~\eqref{eq:fbar_domain1} makes the failure indicator a binary decision variable, and constraint~\eqref{eq:fbar_domain2} keeps the base failure indicator for every pair that propagation cannot change. Setting $\zeta_{\ell i}=0$ for all arcs recovers the base model. In the big-$M$ constant of constraint~\eqref{EXT:r_time_agg}, we replace $\varphi^n_i$ by the earliest failure time $\varphi_i^{n,0}-\bar\zeta_i$.

In the base model and Ext.\ 1, decisions change which tasks fail along a root-to-leaf path, but the probability of each scenario is fixed. In practice, service can also change how likely the next scenario is. In the next section, we capture this effect by letting the assignment at each node select the transition probabilities to its child nodes.

\subsection{
Extension 2: Decision-Dependent Transition Probabilities}
\label{sec:decision_dependent_prob}

We extend the base model~\eqref{EXT} to address decision-dependent transition probabilities. For instance, an inspection crew may identify an equipment fault, while a repair crew can correct it and reduce the risk of overload at the next stage.

We measure the effect of an assignment at node $n$ by its intervention intensity $U^n:=\sum_{i\in\mathcal{T}}\sum_{k\in K}\omega_i\Delta_k x_{ik}^n$, where $\omega_i\in\{1,2,3\}$ measures the criticality of task $i$ and $\Delta_k\in\mathbb Z_{\ge0}$ measures the intervention capability of vehicle $k$. A vehicle that only observes has $\Delta_k=0$. In practice, an operator calibrates transition probabilities for a few ranges of intensity rather than for every value. Thus, we divide the integer intensities into $I+1$ intervention levels with calibrated integer thresholds $0=u_0<u_1<\cdots<u_I\le\bar U_{t(n)}$, where $\bar U_{t(n)}:=\sum_{k\in K}B_{kt(n)}\Delta_k\max_i\omega_i$ bounds the intensity. Level $\iota<I$ contains the intensities from $u_\iota$ to $u_{\iota+1}-1$, and level $I$ contains those from $u_I$ to $\bar U_{t(n)}$. The binary variable $z_\iota^n$ equals one if $U^n$ lies in level $\iota$. At each nonleaf node $n$, level $\iota$ has a calibrated vector $\mathbf{p}_n^{(\iota)}$ of transition probabilities to the child nodes of $n$, with $p^{(\iota)}_{nm}>0$ for all $m\in\mathcal{C}(n)$ and $\sum_{m\in\mathcal{C}(n)}p^{(\iota)}_{nm}=1$. Thus, the transition probability to child $m$ is $p_{nm}(U^n)=\sum_{\iota=0}^{I}z_\iota^np^{(\iota)}_{nm}$.

Because the node probabilities $\mathbb{P}(n)$ in objective~\eqref{EXT:obj} are then products of decision-dependent factors, we write the expected cost recursively. For each nonroot node $n$, the variable $\hat{V}_n$ is the expected cost of the subtree rooted at $n$, conditional on reaching $n$. For each child $m$ of $n$ and level $\iota$, the variable $W_{m\iota}$ equals $\hat{V}_m$ if $n$ selects level $\iota$, and zero otherwise. The constant $\bar{V}_m:=(T-t(m)+1)\sum_{k\in K}\gamma_k+\sum_{i\in\mathcal T}\alpha_i$ bounds $\hat{V}_m$. We formulate Extension 2 (Ext.\ 2) using model~\eqref{EXT2}:
\begin{subequations}\label{EXT2}
\begin{align}
\min_{\substack{\mathbf{x},\mathbf{y},\mathbf{b},\mathbf{e},\mathbf{v},\mathbf{r},\mathbf{q},\\ \mathbf{a},\mathbf{U},\mathbf{z},\hat{\mathbf{V}},\mathbf{W}}}\quad & \sum_{k\in K}\sum_{(i,j)\in\tilde{A}}\delta_{ijk}y^{n_1}_{ijk}+\sum_{i\in\mathcal{T}}\alpha_i\bigl(q^{n_1}_i-q^{n_1^-}_i\bigr)+\sum_{m\in\mathcal{C}(n_1)}\sum_{\iota=0}^{I}p^{(\iota)}_{n_1m}W_{m\iota}
\label{EXT2:obj}\\
\text{s.t.}\quad & \eqref{ABLOCK}\text{--}\eqref{SBLOCK},\ \eqref{EXT:init_v}\text{--}\eqref{EXT:terminal_coverage},\nonumber\\
& U^n=\sum_{i\in\mathcal{T}}\sum_{k\in K}\omega_i\Delta_k x_{ik}^n,\quad \forall n\in\mathcal{N},
\label{eq:intervention_intensity}\\
& \sum_{\iota=0}^{I}z_\iota^n=1,\quad \forall n\in\mathcal{N},
\label{eq:regime_one}\\
& U^n\ge\sum_{\iota=0}^{I}u_\iota z_\iota^n,\quad \forall n\in\mathcal{N},
\label{eq:regime_lb}\\
& U^n\le\sum_{\iota=0}^{I-1}(u_{\iota+1}-1)z_\iota^n+\bar{U}_{t(n)}z_I^n,\quad \forall n\in\mathcal{N},
\label{eq:regime_ub}\\
& z_\iota^n\in\{0,1\},\quad \forall n\in\mathcal{N},\ \iota\in\{0,\ldots,I\},
\label{eq:regime_bin}\\
& \hat{V}_n=\sum_{k\in K}\sum_{(i,j)\in\tilde{A}}\delta_{ijk}y^{n}_{ijk}+\sum_{i\in\mathcal{T}}\alpha_i\bigl(q^{n}_i-q^{n^-}_i\bigr)+\sum_{m\in\mathcal{C}(n)}\sum_{\iota=0}^{I}p^{(\iota)}_{nm}W_{m\iota},\quad \forall n\in\mathcal{N}\setminus\{n_1\},
\label{eq:ext2_value}\\
& W_{m\iota}\le\bar{V}_m z^n_\iota,\quad \forall n\in\mathcal{N},\ m\in\mathcal{C}(n),\ \iota\in\{0,\ldots,I\},
\label{EXT2:mc_z}\\
& W_{m\iota}\le\hat{V}_m,\quad \forall n\in\mathcal{N},\ m\in\mathcal{C}(n),\ \iota\in\{0,\ldots,I\},
\label{EXT2:mc_v}\\
& W_{m\iota}\ge\hat{V}_m-\bar{V}_m\bigl(1-z^n_\iota\bigr),\quad \forall n\in\mathcal{N},\ m\in\mathcal{C}(n),\ \iota\in\{0,\ldots,I\},
\label{EXT2:mc_lb}\\
& W_{m\iota}\ge0,\quad \forall n\in\mathcal{N},\ m\in\mathcal{C}(n),\ \iota\in\{0,\ldots,I\}.
\label{EXT2:mc_nn}
\end{align}
\end{subequations}
The objective~\eqref{EXT2:obj} is the cost at the root plus the expected cost of its child nodes under the level selected at the root. Constraint~\eqref{eq:intervention_intensity} computes the intervention intensity. Constraints~\eqref{eq:regime_one}--\eqref{eq:regime_bin} select exactly one level and place $U^n$ within it. Constraint~\eqref{eq:regime_ub} relies on the integrality of $U^n$, which holds because $\Delta_k$ and $\omega_i$ are integers. Constraint~\eqref{eq:ext2_value} adds the cost at node $n$ to the expected cost of its child nodes, because $\sum_{\iota=0}^{I}p^{(\iota)}_{nm}W_{m\iota}=p_{nm}(U^n)\hat{V}_m$. Constraints~\eqref{EXT2:mc_z}--\eqref{EXT2:mc_nn} are the McCormick inequalities~\citep{mccormick1976computability} of $W_{m\iota}=z^n_\iota\hat{V}_m$. They are exact because $z^n_\iota$ is binary and $0\le\hat V_m\le\bar V_m$. If $z^n_\iota=0$, constraint~\eqref{EXT2:mc_z} sets $W_{m\iota}=0$. If $z^n_\iota=1$, constraints~\eqref{EXT2:mc_v}--\eqref{EXT2:mc_lb} set $W_{m\iota}=\hat{V}_m$. Taking $\mathbf{p}_n^{(\iota)}=\mathbf{p}_n^{(0)}$ for every $\iota$ recovers the base model~\eqref{EXT}.

Models~\eqref{EXT}--\eqref{EXT2} are mixed-integer linear programs (MILPs) over the whole scenario tree, whose number of nodes grows exponentially with the number of stages. In the next section, we decompose the tree by stage and solve one node problem at a time.

\section{An SDDiP-type Algorithm for Multi-stage SMIPs}
\label{sec:snd}

We solve the multi-stage inspection problem with a route-based SDDiP-type algorithm that follows the forward--backward structure of \citet{Zou2019StochasticProgramming}. Models~\eqref{EXT}--\eqref{EXT2} are computationally challenging, because the long-horizon evolution of the state, driven by task failures and service decisions, couples all nodes of the tree. We address this challenge by decomposing the problem by stage and node. Section~\ref{sec:two_horizon} defines the node problem, in which each vehicle selects an inspection route given the state inherited from the parent node, and Section~\ref{sec:node_milp} states it as an MILP. In Section~\ref{subsec:snd_lb_forward}, a forward pass generates decisions along sampled root-to-leaf paths with the current approximations of the cost-to-go functions. A backward pass then adds cutting planes that tighten these approximations. Section~\ref{sec:ncd_ddu} adapts both passes to Ext.\ 1 in~\eqref{EXT1}, and Ext.\ 2 in~\eqref{EXT2}. Section~\ref{subsec:snd_convergence} proves that the algorithm converges in finitely many iterations. Section~\ref{sec:feasibility} then strengthens the node problem with valid inequalities that penalize states without a feasible completion. Finally, Section~\ref{sec:node_policies} describes the policies that the algorithm returns.

\subsection{Route-Based Node Problem}
\label{sec:two_horizon}

Consider a node $n\in\mathcal{N}_t$ in stage $t$ with parent node $n^{-}$. At node $n$, each vehicle selects an inspection route. A route fixes the tasks of the vehicle and their completion times. Thus, it also fixes the failures that the vehicle prevents and the state passed to the child nodes. We define the inherited state at node $n$ as $\mathbf{s}^{n^{-}} := \bigl(\mathbf{v}^{n^{-}}, \mathbf{q}^{n^{-}}, \mathbf{e}^{n^{-}}\bigr)$, where $\mathbf{v}^{n^{-}}$ and $\mathbf{q}^{n^{-}}$ denote the task completion and failure indicators at the parent node. The vector $\mathbf{e}^{n^{-}}$ contains the end-depot indicators of node $n^{-}$, which are the start-depot indicators of stage $t(n)$ in constraint~\eqref{EXT:carry}. 

Figure~\ref{fig:Representations} contrasts the sequential representation with the route-based representation, in which the routing problems separate by vehicle.

\begin{figure}[htbp!]
  \centering
  \begin{subfigure}[b]{0.35\textwidth}
    \centering
    \includegraphics[width=\textwidth]{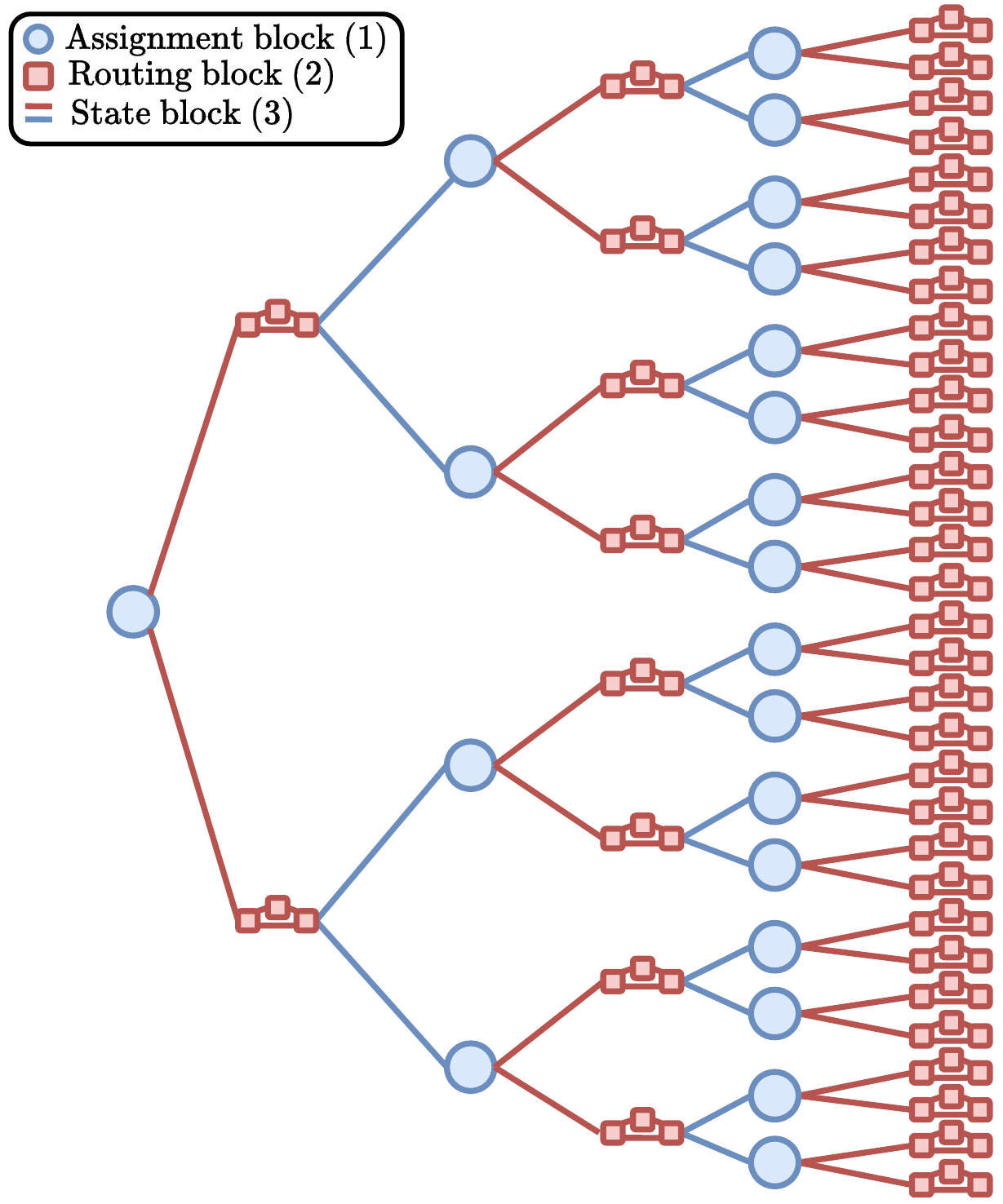}
    \caption{Sequential assignment and routing representation. The nodes of the scenario tree carry assignment decisions, followed by routing-and-timing decisions.}
    \label{fig:multi_stage}
  \end{subfigure}%
  \hspace{50pt}
  \begin{subfigure}[b]{0.40\textwidth}
    \centering
    \includegraphics[width=\textwidth]{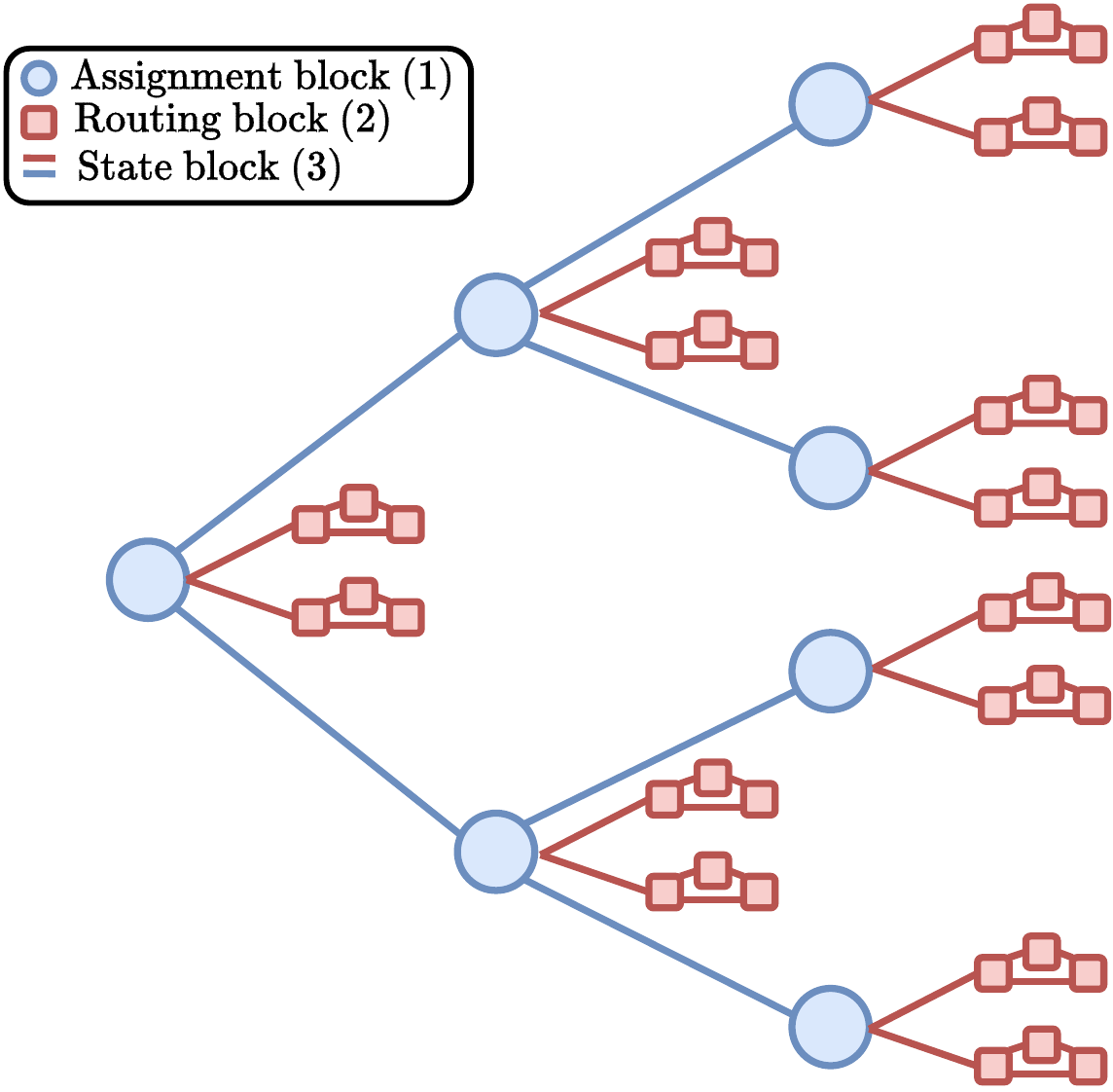}
    \caption{Route-based representation. We choose assignments jointly with the routes that serve the tasks.}
    \label{fig:decomposed_stage}
  \end{subfigure}
  
 \caption{Two representations of the dynamic inspection problem. Representation~\ref{fig:multi_stage} decides the assignment before the routing, so it is exact only when the routing decisions account for the cost-to-go. Representation~\ref{fig:decomposed_stage} selects the routes jointly with the cost-to-go of the outgoing state, as node problem~\eqref{NODE} does.}
  \label{fig:Representations}
\end{figure}

 We call a binary inherited state admissible if it places every vehicle at exactly one depot, that is, $\sum_{d\in\mathcal{D}}e^{n^{-}}_{dk}=1$ for every $k\in K$. The completion and failure indicators may take any binary values. A state can satisfy every constraint of its node problem and still have no feasible completion, because the vehicles cannot serve its open tasks in the remaining stages. We call such a state a dead end and penalize it as an infeasible state. Let $V_{\mathrm{base}}^n(\mathbf{s}^{n^{-}})$ denote the optimal conditional expected cost of the unpenalized subtree of the base model~\eqref{EXT}, with $V_{\mathrm{base}}^n=+\infty$ when $\mathbf{s}^{n^{-}}$ is a dead end. In contrast, $V^n$ denotes the optimal conditional expected cost of the penalized recursion, which permits dummy routes and charges $P$ per task left open at a leaf node. We define $V^n$ by backward induction and keep it finite at every admissible binary inherited state to construct globally valid cuts.

The routing block~\eqref{OP} generates the routes of each node problem. For vehicle $k$ and stage $t$, let $\mathcal{R}_{kt}$ contain the elementary routes of at most $B_{kt}$ tasks that satisfy routing block~\eqref{OP} with their earliest schedules. We write $\mathcal{R}_k$ when the stage is clear. Each nonempty route $\varsigma\in\mathcal{R}_k$ has a start depot $o_\varsigma$, an admissible end depot $d_\varsigma\in\mathcal{D}_k$, a task set $S_\varsigma$, travel cost $\delta_\varsigma$, and completion times $a_{i\varsigma}$ measured from the start of the stage. We also include a dummy route at each depot, with an empty task set, unchanged depot, and objective coefficient $\delta_\varsigma=P$, to keep the node problem feasible. The base model excludes these dummy routes. Given the inherited state, we compute the on-time indicator $\chi^n_{i\varsigma}:=\mathbb{I}\{\theta_{kt(n)}+a_{i\varsigma}\le\varphi_i^n\}$ for each route $\varsigma\in\mathcal{R}_k$ and task $i\in S_\varsigma$. This indicator identifies whether the route can complete task $i$ by its failure time, as required by constraint~\eqref{EXT:r_time_agg}. For small durations $\gamma_k$ and task budgets, we enumerate the route sets. Let $\lambda^n_{k\varsigma}\in\{0,1\}$ indicate whether vehicle $k$ performs route $\varsigma$ at node $n$. 

The penalty $P$ plays the role of a big-$M$ constant, and it must exceed the cost that any feasible completion can save by leaving a task open or a vehicle idle. For the equivalence result, we assume that all physical stage costs are nonnegative and bounded and that a constant $C_{\mathrm{feas}}$ bounds the total cost on every root-to-leaf path under every feasible policy of the base model. The constant $C_{\mathrm{feas}}=T\sum_{k\in K}\gamma_k+\sum_{i\in\mathcal{T}}\alpha_i$ satisfies this condition. We add any further nonnegative cost, such as a depot-relocation cost, to this bound. Let
\begin{equation}
p_{\min}:=\min_{n\in\mathcal N_T}\prod_{(j,m)\text{ on }\Pi(n)}p_{jm},
\label{eq:pmin}
\end{equation}
where $\Pi(n)$ is the root-to-leaf path that ends at leaf $n\in\mathcal N_T$. We choose a penalty $P$ that satisfies
\begin{equation}
P>\frac{C_{\mathrm{feas}}}{p_{\min}}.
\label{eq:penalty_sufficient}
\end{equation}
A root-to-leaf path has $T-1$ arcs. Thus, if every conditional probability is at least $\underline{p}$, then $p_{\min}\ge(\underline{p})^{T-1}$. For example, with three equally likely child nodes at every node and $T=4$, $p_{\min}=(1/3)^3=1/27$, and $P>27\,C_{\mathrm{feas}}$ suffices.

To keep the node problem an MILP, we approximate the value $V^m$ of each child node $m$ from below with cutting planes. Because the inherited state is pure binary, SDDiP generates cuts that are valid at every binary state and exact at the states that the algorithm visits~\citep{Zou2019StochasticProgramming}. Section~\ref{subsec:snd_lb_forward} describes these cuts. For each child $m$, we keep a pool $\mathcal H(m)$ of cuts, each an affine function $\pi_{h,m}^{\top}\mathbf{s}+\pi_{0,h,m}$ of the outgoing state $\mathbf{s}$ of the parent node. Let $V_{\min}$ be a lower bound on every cost-to-go function. The current approximation of $V^m$ is
\begin{equation}
\underline V^m(\mathbf{s}):=\max\Bigl\{V_{\min},\ \max_{h\in\mathcal H(m)}\bigl(\pi_{h,m}^{\top}\mathbf{s}+\pi_{0,h,m}\bigr)\Bigr\},\qquad\forall\,\mathbf{s}\in\{0,1\}^{\dim\mathbf{s}}.
\label{eq:cut_approx}
\end{equation}
Because every cut is valid, $\underline V^m(\mathbf{s})\le V^m(\mathbf{s})$ for every binary state $\mathbf{s}$. Because all costs are nonnegative, we take $V_{\min}=0$ throughout and include the trivial cut with $\pi_{h,m}=\mathbf{0}$ and $\pi_{0,h,m}=0$ in every pool.

\subsection{Node MILP}
\label{sec:node_milp}

We now combine the route constraints and the current cuts. The node problem is
\begin{subequations}\label{NODE}
\begin{align}
\underline Q_n(\mathbf{s}^{n^-}) := \min_{\substack{\mathbf{x},\mathbf{e},\mathbf{v},\mathbf{r},\\ \mathbf{q},\mathbf{w},\bm{\lambda},\vartheta}}\quad& \sum_{k\in K}\sum_{\varsigma\in\mathcal{R}_k}\delta_\varsigma\lambda^n_{k\varsigma}
 + \sum_{i\in\mathcal{T}}\alpha_i\bigl(q_i^n-q_i^{n^{-}}\bigr)
 + P\sum_{i\in\mathcal{T}}w_i^n
 + \sum_{m\in\mathcal{C}(n)}p_{nm}\,\vartheta_m
 \label{NODE:obj}\\
\text{s.t.}\quad&  (\mathbf{v}^n,\mathbf{q}^n)\in\mathcal{S}^n(\mathbf{v}^{n^{-}},\mathbf{q}^{n^{-}},\mathbf{x}^n,\mathbf{r}^n),
 \label{NODE:state}\\
 & \sum_{\varsigma\in\mathcal{R}_k:\,o_\varsigma=d}\lambda^n_{k\varsigma}=e_{dk}^{n^{-}},\quad \forall k\in K,\ d\in\mathcal{D},
 \label{NODE:one}\\
& e_{dk}^{n}=\sum_{\varsigma\in\mathcal{R}_k:\,d_\varsigma=d}\lambda^n_{k\varsigma},\quad \forall k\in K,\ d\in\mathcal{D},
 \label{NODE:end}\\
& x_{ik}^n=\sum_{\varsigma\in\mathcal{R}_k:\,i\in S_\varsigma}\lambda^n_{k\varsigma},\quad\forall i\in\mathcal{T},\ k\in K,
 \label{NODE:assign}\\
& r_{ik}^n\le \sum_{\varsigma\in\mathcal{R}_k:\,i\in S_\varsigma}\chi^n_{i\varsigma}\lambda^n_{k\varsigma},\quad\forall i\in\mathcal{T},\ k\in K,
 \label{NODE:timely}\\
& \vartheta_m\ge\pi_{h,m}^{\top}\mathbf{s}^n+\pi_{0,h,m},\quad \forall m\in\mathcal C(n),\ h\in\mathcal H(m),
 \label{NODE:epi}\\
& v^n_i+w^n_i\ge1,\quad \forall i\in\mathcal{T},\ \text{if } n\in\mathcal{N}_T,
 \label{NODE:term}\\
& w^n_i=0,\quad \forall i\in\mathcal{T},\ \text{if } n\notin\mathcal{N}_T,
 \label{NODE:term0}\\
& \lambda^n_{k\varsigma}\in\{0,1\},\quad\forall k\in K,\ \varsigma\in\mathcal{R}_k,
 \label{NODE:bin}\\
& x_{ik}^n,\,r_{ik}^n\in\{0,1\},\quad\forall i\in\mathcal{T},\ k\in K,
 \label{NODE:binx}\\
& v_i^n,\,q_i^n,\,w_i^n\in\{0,1\},\quad\forall i\in\mathcal{T}.
 \label{NODE:dom}
\end{align}
\end{subequations}
Problem~\eqref{NODE} is the node problem of model~\eqref{EXT} at node $n$ with the inherited state fixed, in which dummy routes and tasks left open at the end of the horizon cost $P$. Constraint~\eqref{NODE:state} imposes the state block~\eqref{SBLOCK}, which preserves the completion and failure updates. The route set enforces the task budget in assignment block~\eqref{ABLOCK} and the routing requirements in routing block~\eqref{OP} for nonempty routes. Constraints~\eqref{NODE:one}--\eqref{NODE:timely} select one route per vehicle and link it to the depots, the assignment, and the on-time completions. Specifically, each route ends at an admissible depot within the task budget and the duration $\gamma_k$. The variables $\vartheta_m$ in constraints~\eqref{NODE:epi} bound the approximated cost of the descendants, and constraint~\eqref{NODE:term} represents the terminal coverage constraint~\eqref{EXT:terminal_coverage}. A dummy route and each terminal coverage violation $w_i^n=1$ incur a nonnegative penalty $P$. The node recursion is therefore feasible at every admissible inherited state.

We obtain the exact value $Q_n(\mathbf{s}^{n^-})$ of the penalized recursion by replacing the cuts~\eqref{NODE:epi} with the exact child values, $\vartheta_m\ge V^m(\mathbf{s}^n)$.

\begin{proposition}[Exactness of the route-based node problem\label{prop:node_exact}]
Under the positive-support assumption and the penalty condition~\eqref{eq:penalty_sufficient}, the penalized recursion satisfies $Q_n(\mathbf{s})=V^n(\mathbf{s})$ at every admissible inherited binary state. Moreover, at every state that is not a dead end,
\[
V^n(\mathbf{s})=V_{\mathrm{base}}^n(\mathbf{s}),
\]
and every optimal penalized policy uses only nonempty routes and satisfies terminal coverage throughout that subtree. In particular, if model~\eqref{EXT} is feasible, the root value $Q_{n_1}(\mathbf{s}^{n_1^{-}})$ equals its optimal value.
\end{proposition}
The proof of Proposition~\ref{prop:node_exact} is in Appendix~\ref{ec:proofs}. With the node problem stated as an MILP, we now describe how SDDiP obtains bounds. A forward pass gives upper bounds (UB), the root node problem gives lower bounds (LB), and a backward pass generates the cutting planes that tighten the lower bounds.

\subsection{Forward and Backward Passes}
\label{subsec:snd_lb_forward}

In the forward pass, we solve node problem~\eqref{NODE} with the current cuts at the inherited state. The root value is a lower bound on the penalized recursion. At a feasible root, it also bounds the optimal value of the base model under Proposition~\ref{prop:node_exact}. The expected cost of a complete policy gives an upper bound for the penalized recursion and also for the base model if the policy incurs no penalties. The root lower bound is nondecreasing as cuts are added to a fixed formulation.

In the backward pass, each visited node $m$ with trial state $\hat{\mathbf{s}}$ returns two cuts to its parent node: a strengthened Benders cut and an integer optimality cut. Both cuts rely on relaxing the inherited state. We replace the inherited state of node problem~\eqref{NODE} at $m$ by a local copy $\tilde{\mathbf{s}}=(\tilde{\mathbf{v}},\tilde{\mathbf{q}},\tilde{\mathbf{e}})$ and add the copy constraint $\tilde{\mathbf{s}}=\hat{\mathbf{s}}$. Dualizing or relaxing this copy constraint yields cuts that are valid at every binary state. Figure~\ref{fig:SDDiP_flow} illustrates the forward and backward passes at a node.

The strengthened Benders cut of \citet{Zou2019StochasticProgramming} starts from the LP relaxation of node problem~\eqref{NODE}, whose optimal duals $\pi$ of the copy constraint give the slope of the cut. It then solves the Lagrangian subproblem with the copy constraint relaxed, which gives the intercept:
\begin{equation}
    V^m(\mathbf{s})\ \ge\ \pi^\top\mathbf{s}+\mathcal{L}_m(\pi),\label{eq:sb_cut}
\end{equation}
where
\begin{subequations}\label{eq:Vm_lag_dual}
\begin{align}
\mathcal{L}_m(\pi):=\min_{\substack{\mathbf{x},\mathbf{e},\mathbf{v},\mathbf{r},\\
\mathbf{q},\mathbf{w},\bm{\lambda},\vartheta,\tilde{\mathbf{s}}}} \quad & \sum_{k\in K}\sum_{\varsigma\in\mathcal{R}_k}\delta_\varsigma\lambda^m_{k\varsigma} + \sum_{i\in\mathcal{T}}\alpha_i\bigl(q_i^m-\tilde{q}_i\bigr) + P\sum_{i\in\mathcal{T}}w_i^m + \sum_{m'\in\mathcal{C}(m)}p_{mm'}\,\vartheta_{m'}-\pi^\top\tilde{\mathbf{s}}
\label{LAG:obj}\\
\mbox{s.t.}\quad& \text{state block }\eqref{SBLOCK},\ \eqref{NODE:one}\text{--}\eqref{NODE:timely},\ \eqref{NODE:term}\text{--}\eqref{NODE:dom}\text{ at node } m,\text{ with }\tilde{\mathbf{s}}\text{ in place of }\mathbf{s}^{m^{-}},\nonumber\\
& \vartheta_{m'}\ge\pi_{h,m'}^{\top}\mathbf{s}^m+\pi_{0,h,m'},\quad \forall m'\in\mathcal{C}(m),\ h\in\mathcal H(m'),\label{LAG:cuts}\\
& \tilde{\mathbf{s}}\in\{0,1\}^{\dim\mathbf{s}}. \label{LAG:binary}
\end{align}
\end{subequations}
Constraint~\eqref{LAG:cuts} keeps the current cuts of the child nodes of $m$, and constraint~\eqref{LAG:binary} keeps the local copy binary, which makes the cut valid at every binary state. In the LP phase of Algorithm~\ref{alg:SND_twohorizon}, we use the optimal LP value $\underline{Q}^{\mathrm{LP}}_m$ at $\hat{\mathbf{s}}$ in place of the Lagrangian intercept, which gives $V^m(\mathbf{s})\ge \underline{Q}^{\mathrm{LP}}_m+\pi^\top(\mathbf{s}-\hat{\mathbf{s}})$.

The integer optimality cut of \citet{Zou2019StochasticProgramming} uses the value of node problem~\eqref{NODE} at the trial state instead of a dual. Let $\underline{Q}_m(\hat{\mathbf{s}})$ be the value of node problem~\eqref{NODE} at $m$ and $\hat{\mathbf{s}}$ with the current approximations of its child nodes. With $V_{\min}=0$, the cut
\begin{equation}
V^m(\mathbf{s})\ \ge\ \underline{Q}_m(\hat{\mathbf{s}})\Bigl(1-\sum_{j}|s_j-\hat{s}_j|\Bigr)
\label{eq:ioc_full}
\end{equation}
is valid at every binary $\mathbf{s}$ and tight at $\hat{\mathbf{s}}$. It is weak away from $\hat{\mathbf{s}}$, because its right-hand side is at most zero as soon as one coordinate differs. We refine it by summing only over the state coordinates that affect the value. The following proposition identifies them, and Appendix~\ref{ec:proofs} gives its proof.

\begin{figure}[htbp!]
    \centering
    \includegraphics[width=0.55\textwidth]{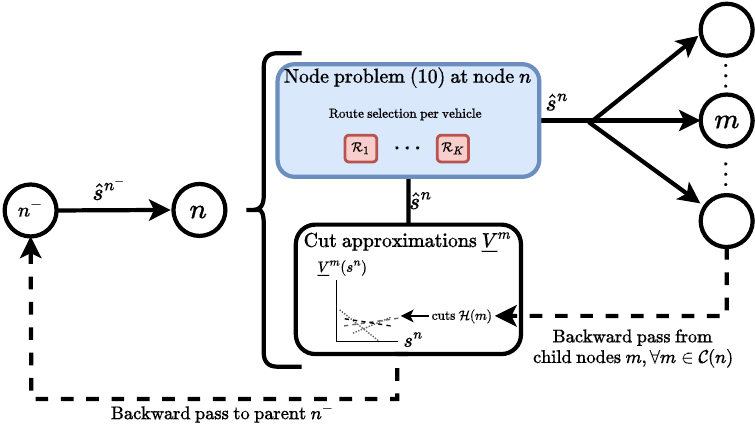}
    \caption{Detail of SDDiP at node $n$. 
Solid arrows denote the forward pass. We decide the assignment and routing using the current approximations $\underline{V}^{m}$ of child nodes $m \in \mathcal{C}(n)$.
Dashed arrows denote the backward pass. SDDiP updates the cut pool $\mathcal{H}(m)$ to tighten the approximations of the cost-to-go functions using feedback from the forward pass.}
    \label{fig:SDDiP_flow}
\end{figure}

\begin{proposition}[Relevant coordinates\label{prop:inert}]
At a state with $v_i=1$, $V^m$ does not depend on $q_i$. Consequently, cut~\eqref{eq:ioc_full} remains valid at every binary $\mathbf{s}$ when its sum excludes these coordinates.
\end{proposition}

By Proposition~\ref{prop:inert}, the restricted integer optimality cut
\begin{equation}
V^m(\mathbf{s})\ \ge\ \underline{Q}_m(\hat{\mathbf{s}})\Bigl(1-\sum_{j\in J(\hat{\mathbf{s}})}|s_j-\hat{s}_j|\Bigr)
\label{eq:ioc_cut}
\end{equation}
is valid at every binary $\mathbf{s}$. Here $J(\hat{\mathbf{s}})$ contains the relevant coordinates of the state. It excludes $q_i$ for every task with $\hat{v}_i=1$. Thus, cut~\eqref{eq:ioc_cut} stays tight at $\hat{\mathbf{s}}$ and is at least as strong as cut~\eqref{eq:ioc_full} at every other state. The strengthened Benders cut need not be tight at $\hat{\mathbf{s}}$, and thus the two cuts complement each other.

\subsection{Extensions to the Base Model in SDDiP}
\label{sec:ncd_ddu}
As in the base model~\eqref{EXT}, for Ext.\ 1 in~\eqref{EXT1}, and Ext.\ 2 in~\eqref{EXT2}, dummy routes and tasks left open at the end of the horizon cost $P$. Under Ext.\ 2 in~\eqref{EXT2}, definition~\eqref{eq:pmin} uses on each arc the smallest probability across intervention levels, $\min_{0\leq\iota\leq I} p^{(\iota)}_{jm}$, which bounds the probability of the arc from below under any assignment. Proposition~\ref{prop:node_exact} also applies to Ext.\ 1 in~\eqref{EXT1}, and Ext.\ 2 in~\eqref{EXT2}, because $C_{\mathrm{feas}}$ bounds their physical costs and every intervention-level probability satisfies the positive-support assumption.

\paragraph{Failure propagation.}
For Ext.\ 1 in~\eqref{EXT1}, we add constraint~\eqref{eq:propagation_trigger} to node problem~\eqref{NODE} and include $\mathbf O^n$ in the outgoing state. Because the trigger constraint~\eqref{eq:propagation_trigger} depends on the completion indicators rather than on the completion times, $\mathbf O^n$ is binary. Thus, the cuts of Section~\ref{subsec:snd_lb_forward} apply to it. In addition, $V^m$ does not depend on $O_i$ at a state in which the vehicles have served every task $j$ with $i\in\operatorname{Pred}(j)$, and thus the set $J(\hat{\mathbf{s}})$ of cut~\eqref{eq:ioc_cut} also excludes these coordinates (see the proof of Proposition~\ref{prop:inert} in Appendix~\ref{ec:proofs}). The inherited propagation state determines the effective failure indicators $\bar f_i^n$ and the on-time indicators $\chi_{i\varsigma}^n$. When we generate cuts, we link these indicators to the local copy of the inherited propagation state using constraints of the form~\eqref{eq:propagation_effective}--\eqref{eq:fbar_link_lb}.

\paragraph{Decision-dependent transitions.}
For Ext.\ 2 in~\eqref{EXT2}, we add constraints~\eqref{eq:intervention_intensity}--\eqref{eq:regime_bin} to node problem~\eqref{NODE}, and its continuation term becomes $\sum_{m\in\mathcal{C}(n)}\sum_{\iota=0}^{I}p^{(\iota)}_{nm}z_\iota^n\vartheta_m$. The value $V^m$ of a child node $m$ is the expected cost conditional on reaching $m$, and thus it does not depend on the transition probability vector that node $n$ selects. Therefore, every cut of $V^m$ remains valid, and decision dependence enters only through the products $z^n_\iota\vartheta_m$. To linearize them, we split the outgoing state and the child values by intervention level, $\mathbf{s}^n=\sum_{\iota=0}^{I}\mathbf{s}^n_\iota$ and $\vartheta_m=\sum_{\iota=0}^{I}\vartheta_{m\iota}$. The continuation term becomes $\sum_{m\in\mathcal{C}(n)}\sum_{\iota=0}^{I}p^{(\iota)}_{nm}\vartheta_{m\iota}$, and we replace each cut $\vartheta_m\ge\pi_{h,m}^\top\mathbf{s}^n+\pi_{0,h,m}$, $h\in\mathcal H(m)$, in constraints~\eqref{NODE:epi} with its perspective form
\begin{subequations}\label{eq:perspective_cut}
\begin{align}
& \vartheta_{m\iota}\ge\pi_{h,m}^\top\mathbf{s}^n_\iota+\pi_{0,h,m}\,z^n_\iota,
\quad \forall m\in\mathcal C(n),\ h\in\mathcal H(m),\ \iota\in\{0,\ldots,I\},
\label{eq:perspective_epi}\\
& \mathbf{s}^n_\iota\ge\mathbf{0},
\quad \forall \iota\in\{0,\ldots,I\},
\label{eq:perspective_lb}\\
& \mathbf{s}^n_\iota\le z^n_\iota\mathbf{1},
\quad \forall \iota\in\{0,\ldots,I\}.
\label{eq:perspective_ub}
\end{align}
\end{subequations}

\begin{proposition}[Perspective cuts\label{prop:perspective}]
For binary $\mathbf{z}^n$ with $\sum_{\iota=0}^{I}z^n_\iota=1$ and selected level $\hat{\iota}$, constraints~\eqref{eq:perspective_cut} imply $\mathbf{s}^n_{\hat{\iota}}=\mathbf{s}^n$, $\vartheta_{m\hat{\iota}}\ge\pi_{h,m}^\top\mathbf{s}^n+\pi_{0,h,m}$ for every $h\in\mathcal H(m)$, and $\vartheta_{m\iota}\ge0$ for every $\iota\neq\hat{\iota}$. Therefore, the continuation term equals $\sum_{m\in\mathcal{C}(n)}p^{(\hat{\iota})}_{nm}\underline{V}^m(\mathbf{s}^n)$ at an optimal solution.
\end{proposition}

The statement follows by substitution, taking the minimum feasible epigraph values of the child nodes. The perspective construction follows the disjunctive formulation of \citet{balas1998disjunctive} and the perspective reformulations with indicator variables studied by \citet{gunluk2010perspective}.

\subsection{Finite Convergence}
\label{subsec:snd_convergence}

\begin{theorem}[Finite convergence\label{thm:snd_convergence}]
Assume a finite scenario tree, finite route sets, and a binary inherited state. Assume further that SDDiP solves every node problem and every subproblem of the integer optimality cuts to global optimality, that every generated cut is valid, and that, for some finite $\nu\ge1$, every $\nu$-th iteration traverses the whole tree. Then SDDiP terminates after finitely many complete-tree traversals with $\mathrm{LB}=\mathrm{UB}$ equal to the optimal value of the penalized recursion. If, in addition, the base model is feasible and the conditions of Proposition~\ref{prop:node_exact} hold, then this value is the optimal value of model~\eqref{EXT}, and the resulting policy is feasible for model~\eqref{EXT}.
\end{theorem}

The proof (see Appendix~\ref{ec:proofs}) follows the finite-termination argument of \citet{Zou2019StochasticProgramming}. Each complete traversal either terminates with an optimal policy or makes exact a new pair of a node and a visited binary state. Thus, the number of nonterminating complete traversals is at most the number of such pairs. By Proposition~\ref{prop:perspective}, Theorem~\ref{thm:snd_convergence} also applies to the node problem with the perspective cuts of Ext.\ 2 in~\eqref{EXT2}. We present our implementation of SDDiP in Algorithm~\ref{alg:SND_twohorizon}. SDDiP stops when $\mathrm{UB}-\mathrm{LB}\le10^{-6}\max\{1,|\mathrm{UB}|\}$, when it reaches the time limit, or when it reaches the iteration limit. To break ties among optimal solutions of a node problem, we select one that maximizes the number of on-time completions $\sum_{i\in\mathcal{T}}\sum_{k\in K}r^n_{ik}$.

\begin{algorithm}[t]
\small
\caption{SDDiP-type algorithm for multi-stage inspection routing}
\label{alg:SND_twohorizon}
\begin{algorithmic}[1]
\State \textbf{Initialization:}
$i \gets 1$, initialize cut pools
$\mathcal{H}(m)\gets\{V^m\ge0\}$ for all $m\in\mathcal{N}$.
\State \textbf{Phase 1 (LP):} repeat forward and backward passes with the LP relaxations of node problem~\eqref{NODE} and the cuts from their duals until no child-node value at its trial state exceeds the parent's epigraph estimate by more than $10^{-7}$ in relative terms, or we reach an LP-phase limit.
\While{stopping criterion not satisfied}
\State Select the root-to-leaf paths $\Omega^i$. Every $\nu$-th iteration, let $\Omega^i$ contain all paths of the tree.
\State \textbf{Forward pass}
\For{each node $n$ on a path of $\Omega^i$, in increasing order of stage}
\State Solve node problem~\eqref{NODE} using the current approximations at the inherited state passed by $n^{-}$. Store $\hat{\mathbf{x}}_{n}^{i}$ and the state $\hat{\mathbf{s}}_m^i$ that it passes to each child node $m\in\mathcal{C}(n)$.
\EndFor
\State \textbf{Upper bound update:} if $\Omega^i$ contains all paths, evaluate the SDDiP policy on the tree and update $\mathrm{UB}$.
\State \textbf{Backward pass}
\For{$t=T,\dots,2$}
\For{$m\in\mathcal{N}_{t}$ on a path of $\Omega^i$}
\State Generate a strengthened Benders cut and an integer optimality cut at $\hat{\mathbf{s}}_{m}^i$ and add them to $\mathcal{H}(m)$.
\EndFor
\EndFor
\State \textbf{Lower bound update:} Compute the root lower bound by solving node problem~\eqref{NODE} at the root with the current cuts. 
\State Set $i\gets i+1$.
\EndWhile
\end{algorithmic}
\end{algorithm}

\subsection{
Feasibility and Valid Inequalities}
\label{sec:feasibility}

The terminal coverage condition in constraints~\eqref{EXT:terminal_coverage} couples all stages of the problem. A dead end incurs a penalty on at least one descendant path, and thus its conditional expected cost is at least $p_{\min}P$. The cuts of SDDiP eventually detect this cost, but until they do, the forward pass can reach dead ends repeatedly. We therefore add valid inequalities that penalize dead ends earlier.

A group is a pair $(K_g,\mathcal{T}_g)$ such that the vehicles of $K_g$ serve only tasks of $\mathcal{T}_g$, and only these vehicles serve the tasks of $\mathcal{T}_g$. The fleet with all tasks is always a group. When the vehicles of each district serve only its tasks, every district is also a group. Let $\rho^n_g:=|K_g|(T-t(n))$ be the number of routes of the vehicles of group $g$ in the stages after node $n$. Every vehicle serves at least one task in each stage by constraint~\eqref{EXT:active_vehicle} and at most its budget. Thus, the open tasks of the group satisfy the reachability inequalities
\begin{subequations}\label{eq:reach}
\begin{align}
& \sum_{i\in\mathcal{T}_g}(1-v_i^n)\ \ge\ \rho^n_g,
\label{eq:reach_lb}\\
& \sum_{i\in\mathcal{T}_g}(1-v_i^n)\ \le\ \sum_{k\in K_g}\sum_{\tau=t(n)+1}^{T}B_{k\tau}.
\label{eq:reach_ub}
\end{align}
\end{subequations}
These inequalities count tasks but not routes, and therefore a state can satisfy them and still be a dead end. Let $\mathcal{G}_g$ contain the task sets of the nonempty routes of the vehicles of $K_g$, over all stages and depots. The routable-assignment certificate accepts a state for group $g$ if the open tasks $\{i\in\mathcal{T}_g:v_i^n=0\}$ split into exactly $\rho^n_g$ sets of $\mathcal{G}_g$. Checking this condition is a small exact-cover problem. If the forward pass at a node of stage $t$ passes completion indicators $\hat{\mathbf{v}}$ that the certificate rejects for group $g$, we add the no-good cut
\begin{equation}
\sum_{i\in\mathcal{T}_g:\,\hat{v}_i=1}(1-v_i^n)+\sum_{i\in\mathcal{T}_g:\,\hat{v}_i=0}v_i^n+\xi^n_g\ \ge\ 1,\quad \forall n\in\mathcal{N}_t,
\label{eq:nogood}
\end{equation}
to every node problem of stage $t$. Let $\bm{\xi}^n$ collect the nonnegative integer slacks of inequalities~\eqref{eq:reach} and the binary slacks $\xi^n_g$ of cut~\eqref{eq:nogood}. We add inequalities~\eqref{eq:reach} and cut~\eqref{eq:nogood} with these slacks to node problem~\eqref{NODE} and the term $P\mathbf{1}^\top\bm{\xi}^n$ to objective~\eqref{NODE:obj}. Thus, node problem~\eqref{NODE} stays feasible at every inherited state, and any violation costs at least $P$. We penalize dead ends rather than exclude them, because the cuts of SDDiP require a cost-to-go function that is finite at every binary state. The Lagrangian subproblem~\eqref{eq:Vm_lag_dual} keeps these inequalities and their penalties.

Finally, a lower bound on the travel cost of the remaining stages gives a coverage cut. For multipliers $\boldsymbol\beta\in[0,P]^{|\mathcal{T}|}$,
\begin{equation}
\vartheta_m\ \ge\ \sum_{\tau=t(n)+1}^{T}\sum_{k\in K}\min\Bigl\{0,\ \min_{\varsigma\in\mathcal{R}_{k\tau}:\,S_\varsigma\neq\emptyset}\Bigl(\delta_\varsigma-\sum_{i\in S_\varsigma}\beta_i\Bigr)\Bigr\}+\sum_{i\in\mathcal{T}}\beta_i(1-v_i^n),\qquad\forall m\in\mathcal{C}(n),
\label{eq:coverage_cut}
\end{equation}
where $\mathcal{R}_{k\tau}$ contains the routes of vehicle $k$ in stage $\tau$ from every admissible start depot. We choose $\boldsymbol\beta$ once by a subgradient method on the root bound and add cut~\eqref{eq:coverage_cut} to every node problem before the first iteration. Under Ext.\ 2 in~\eqref{EXT2}, we impose cut~\eqref{eq:coverage_cut} in the perspective form~\eqref{eq:perspective_cut}, like every other cut.

\begin{proposition}[Valid inequalities\label{prop:valid}]
(i) A state that violates inequalities~\eqref{eq:reach} for some group, or that the certificate rejects for some group, is a dead end. If the vehicles of group $g$ are identical, home-based at a common depot, and have the same routes in every stage, every state that the certificate accepts for $g$ has a completion that serves the open tasks of $g$. (ii) Inequality~\eqref{eq:coverage_cut} is valid at every binary state for every $\boldsymbol\beta\in[0,P]^{|\mathcal{T}|}$. (iii) Adding inequalities~\eqref{eq:reach}--\eqref{eq:coverage_cut} to node problem~\eqref{NODE} does not change $Q_n$ at any state from which the subtree of $n$ admits a feasible policy, and Proposition~\ref{prop:node_exact} and Theorem~\ref{thm:snd_convergence} hold with them.
\end{proposition}

The proof of Proposition~\ref{prop:valid} is in Appendix~\ref{ec:proofs}.

\subsection{Extended Formulation Solutions and SDDiP Policies}
\label{sec:node_policies}

An optimal solution of the extended formulation~\eqref{EXT} fixes the decisions of every node of the scenario tree in advance. In the simulator, this solution executes the decisions of the node that matches the observed regimes. The SDDiP cut pools define the SDDiP policy $\pi[\underline{V}]$. At the node $n$ that matches the observed history, this policy solves node problem~\eqref{NODE} at the observed inherited state $\mathbf{s}$ with $\vartheta_m\ge\underline{V}^m(\mathbf{s}^n)$. When the bounds coincide under the assumptions of Theorem~\ref{thm:snd_convergence}, the SDDiP policy is optimal on the modeled tree. At other states, the cuts remain valid but may be loose. A time or iteration limit instead returns the SDDiP policy of the last iteration and a lower bound. Each SDDiP iteration solves an MILP at every node of the sampled paths, and on larger trees this effort delays good policies. In the next section, we learn an assignment policy that avoids this repeated optimization.

\section{Reinforcement Learning Approach}
\label{sec:rl}

This section builds on the route-based decomposition of Section~\ref{sec:two_horizon}, which we introduced for SDDiP. We aim to learn assignment decisions from simulated system trajectories using the binary operational state and the observed risk information. To this effect, we propose a hierarchical actor-critic reinforcement learning approach to decide each assignment sequentially under feasibility masks. The learned assignment policy provides an alternative to the repeated node optimization that SDDiP requires.

\subsection{MDP Formulation}
\label{subsec:rl_mdp}

We model the stage-wise assignment decisions as a finite-horizon Markov decision process over stages $t\in\{1,\dots,T\}$. At the beginning of stage $t$, the decision maker observes the state
\[
\mathbf{s}_{t-1}:=\bigl(\mathbf{v}_{t-1},\mathbf{q}_{t-1},\mathbf{e}_{t-1},\boldsymbol{\Phi}_{t-1}\bigr),
\]
where $\mathbf{v}_{t-1}$ and $\mathbf{q}_{t-1}$ are the task completion and failure indicators. The vector $\mathbf{e}_{t-1}$ gives the depot of each vehicle at the start of stage $t$, as in constraint~\eqref{EXT:carry}, and $\boldsymbol{\Phi}_{t-1}$ contains risk information about the failure model. In our implementation, $\boldsymbol{\Phi}_{t-1}$ contains the stage, the accumulated stress of Section~\ref{sec:exp_tree}, and the failure times of the node of stage $t$, which the operator observes at the start of stage $t$ and which summarize the observed history of the scenario tree.

An action is a binary assignment $\mathbf{x}_t:=\{x_{ik,t}\}_{i\in\mathcal{T},k\in K}$ constrained by per-vehicle capacity and completion eligibility. Let $n\in\mathcal{N}_t$ be the node of stage $t$. Then $\mathbf{v}^{n^{-}}=\mathbf{v}_{t-1}$, $\mathbf{q}^{n^{-}}=\mathbf{q}_{t-1}$, and $\mathbf{e}^{n^{-}}=\mathbf{e}_{t-1}$, so that the MDP state augments the inherited operational state with $\boldsymbol{\Phi}_{t-1}$. The feasible actions are
\begin{equation}
\begin{aligned}
\mathcal{F}_t:=\bigl\{\mathbf{x}^n\in\mathcal{A}^n:\ &\text{\eqref{EXT:once}, \eqref{NODE:one}, and \eqref{NODE:assign} hold for some}\\
&\text{binary selection }\boldsymbol\lambda^n\text{ of nonempty routes}\bigr\}.
\end{aligned}
\label{eq:rl_action_feas}
\end{equation}
The policy chooses $\mathbf{x}_t\in\mathcal{F}_t$, and the routing oracle described next chooses the route of each vehicle.

Given $(\mathbf{s}_{t-1},\mathbf{x}_t)$, the environment evaluates the routing of the assignment by querying an oracle consistent with the route-based node problem. The oracle selects, for each vehicle, a route that minimizes the immediate travel and failure cost of its assigned task set. Unlike the SDDiP node problem, this oracle does not explicitly optimize the cost-to-go associated with route completion times and end-depot locations. Consequently, the hierarchical actor-critic approach searches a restricted hierarchical policy class. The oracle returns (i) the routing cost $R_t(\mathbf{x}_t;\mathbf{e}_{t-1})$, (ii) completion times $\{a_{ik,t}\}$ that determine whether the vehicles complete the assigned inspections before their failure times, and (iii) the next depot state $\mathbf{e}_t$.

We define the one-step cost as
\begin{equation}
C_t(\mathbf{s}_{t-1},\mathbf{x}_t)
:=
R_t(\mathbf{x}_t;\mathbf{e}_{t-1})+\sum_{i\in\mathcal{T}}\alpha_i\bigl(q_{i,t}-q_{i,t-1}\bigr)
+P\cdot\mathbb{I}\{t=T\}\sum_{i\in\mathcal{T}}(1-v_{i,T}),
\label{eq:rl_stage_cost}
\end{equation}
matching the travel and failure cost structure in objective~\eqref{EXT:obj}, with the penalty $P$ of problem~\eqref{NODE} per task left open. The routing solution and the endogenous failure realization in stage $t$ determine the state transition from $\mathbf{s}_{t-1}$ to $\mathbf{s}_t$. We seek a feasible policy $\pi$ that minimizes $\mathbb{E}_{\pi}[\sum_{t=1}^{T} C_t(\mathbf{s}_{t-1},\mathbf{x}_t)]$.

\subsection{Policy and Training}
\label{subsec:rl_actor_critic}

We interpret RL as simulation-based approximate dynamic programming over the stages. We restrict attention to policy classes that generate assignments feasible for the current stage by construction. Figure~\ref{fig:rl_arch} shows the architecture. Following the hierarchical structure of the problem~\citep{vezhnevets2017feudal}, the actor builds $\mathbf{x}_t$ one task--vehicle pair at a time. Its network scores every pair allowed by a mask and a stop action, and the actor samples the next choice from the softmax over these candidates, as in attention-based construction policies for routing~\citep{kool2019attention}. The actor updates the mask after every choice to keep $\mathbf{x}_t$ in $\mathcal F_t$~\eqref{eq:rl_action_feas}, and thus every sampled assignment is feasible at its stage. Because the mask cannot ensure that the vehicles can route the tasks left open in later stages, a policy can still reach a dead end, incurring a feasibility penalty. When we use the certificate of Section~\ref{sec:feasibility}, we resample every assignment that the certificate rejects (see Appendix~\ref{ec:rl} for implementation details).

\begin{figure}[htbp]
    \centering
    \includegraphics[width=0.65\textwidth]{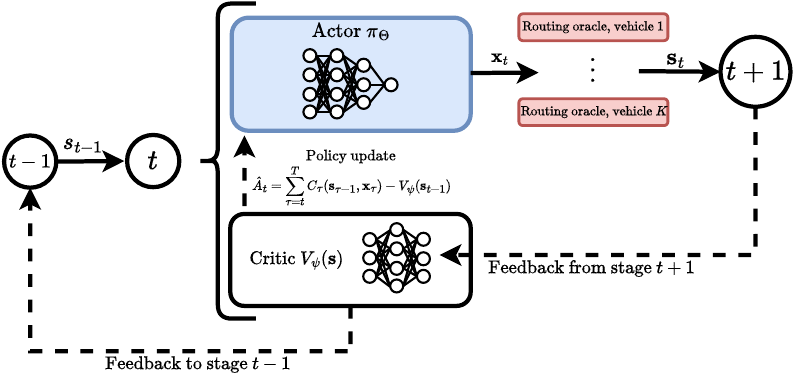}
    \caption{
    An overview of the hierarchical actor-critic reinforcement learning approach. Solid arrows represent the simulation flow. Dashed arrows represent parameter updates after each iteration.}
    \label{fig:rl_arch}
\end{figure}

The critic approximates the cost-to-go of the current policy, $V_\psi(\mathbf{s})\approx\mathbb{E}_{\pi_\Theta}\bigl[\sum_{\tau=t}^TC_\tau(\mathbf{s}_{\tau-1},\mathbf{x}_\tau)\mid\mathbf{s}_{t-1}=\mathbf{s}\bigr]$. For a simulated trajectory, the advantage is the Monte Carlo return minus the critic, $\hat{A}_t:=\sum_{\tau=t}^{T}C_\tau(\mathbf{s}_{\tau-1},\mathbf{x}_\tau)-V_\psi(\mathbf{s}_{t-1})$, and, without certificate resampling, we update the actor against the gradient of the expected total cost $\mathcal{J}(\Theta)$ of policy $\pi_\Theta$. Let $\vec{\mathbf{x}}_t$ denote the actor's ordered sequence of choices, including the stop action, that produces $\mathbf{x}_t$. Then
\begin{equation}
\nabla_{\Theta} \mathcal{J}(\Theta)
=
\mathbb{E}\left[
\sum_{t=1}^{T}
\nabla_{\Theta}\log \pi_{\Theta}(\vec{\mathbf{x}}_t\mid \mathbf{s}_{t-1})\,
\hat{A}_t
\right].
\label{eq:policy_gradient}
\end{equation}
Here $\log\pi_\Theta(\vec{\mathbf{x}}_t\mid\mathbf{s}_{t-1})$ is the sum of the conditional log-probabilities of the sampled choices. Different choice sequences can produce the same assignment. The critic minimizes $\mathbb{E}[\sum_{t=1}^T\hat{A}_t^2]$. With certificate resampling, we use the proposal-score update described in Appendix~\ref{ec:rl}. Algorithm~\ref{alg:rl} in Appendix~\ref{ec:rl} summarizes the training loop. The extended formulation, SDDiP, and RL thus give three ways to obtain a policy. In the next section, we compare their cost and run time on the scenario tree and in a common simulator.

\section{Numerical Studies}
\label{sec:experiments}

This section evaluates the computational performance and solution quality of the proposed stochastic optimization and RL approaches on instances constructed from the Sioux Falls benchmark network and from the storm-water drainage network of the same city.

\subsection{Experimental Setup}
\label{sec:exp_setup}

We compare three approaches. First, we solve the extended formulation of the multi-stage SMIP directly with an off-the-shelf solver (``EXT''). Second, we apply Algorithm~\ref{alg:SND_twohorizon} (``SDDiP''). Third, we use the hierarchical actor-critic reinforcement learning approach in Section~\ref{sec:rl} (``HAC-RL''). We use Julia~1.12, JuMP, and Gurobi~13.0 on an Apple M2 with 16~GB of RAM. We report exact expected costs on the scenario tree and realized costs on simulator trajectories. In the simulator, SDDiP policies act on the observed state as described in Section~\ref{sec:node_policies}.

For every model, we cap SDDiP at 1,000 LP iterations and 1,000 integer passes. On four-stage instances, the LP phase gets at most half the budget and each integer pass traverses the full tree. On six-stage instances, the LP cap is the smaller of 1,800 seconds and half the budget. We traverse the full tree every fifth integer pass, or earlier if the lower bound has improved by less than $10^{-4}$ relative to the last full pass. Other passes sample 10 root-to-leaf paths. We update the root lower bound on every pass. Every SDDiP run starts from an initial feasible solution. For each district, we approximately solve a set-partitioning problem that selects one route for every vehicle of the district and every stage, so that the routes serve every task of the district exactly once at minimum total travel time. A sequential local search selects the routes one at a time and keeps the cheapest feasible selection that it finds. Every node of a stage executes the same routes, and thus this initial solution is a static policy.

\subsection{Instance Setup}
\label{sec:exp_instances}
The instances use the Sioux Falls road and storm-sewer networks~\citep{transportationnetwork,siouxfallsgis}. The instances form three families, and each family answers a different question (Table~\ref{tab:instances}). The road instances R1, R2, and R3 have four stages. SDDiP certifies an optimal solution of almost every model on these instances within 1,800 seconds (Section~\ref{sec:exp_nd}). Thus, we use them to compare SDDiP with EXT, to study repositioning, and to measure the gap of HAC-RL to the optimum. Instance R1 has 21 tasks and one depot, and instances R2 and R3 connect two and three copies of it by highways (Figure~\ref{fig:instances}). Road travel times are proportional to arc length, and service times are scaled for a budget of three tasks per vehicle per stage, which represents a 12-hour shift. The drainage instances D1 and D2 have six stages and one and two districts. Instance D2i is instance D2 with independent districts, in which the vehicles of each district serve only its tasks. We use these three instances to test scalability. We evaluate the models in the simulator on instance DS, which has 12 tasks of the storm-sewer network and four stages. On this instance, we can visualize how failures propagate through each cluster and which tasks the policy of each model protects (Figure~\ref{fig:failure_map}). Moreover, SDDiP solves every model to optimality in every calibration setting, so that differences in realized cost come from the model and not from the solver. Each sample of instance DS draws the nominal lives of its tasks, and the simulator draws the weather in continuous time (Section~\ref{sec:exp_tree}). The drainage instances use the same task budget, 1.5-hour service times, and a travel speed of 32~km/h.

\begin{figure}[t]
\centering
\begin{subfigure}[b]{0.40\textwidth}\centering
\includegraphics{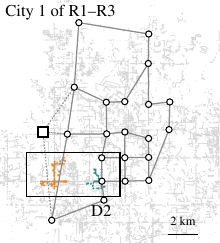}
\caption{Sioux Falls}\end{subfigure}\hfill
\begin{subfigure}[b]{0.58\textwidth}\centering
\includegraphics{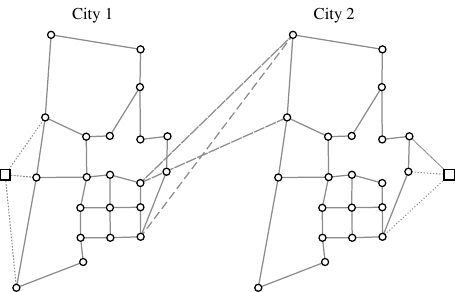}
\caption{Road instance R2}\end{subfigure}
\vspace{1mm}
\includegraphics{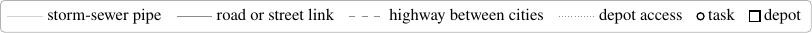}
\caption{Instances on the map of Sioux Falls. (a) The storm-sewer network (gray), the road network with the 21 tasks and the depot of city~1 of instances R1--R3, and the area of instance D2 (box). (b) Instance R2: city~1 of (a) and a copy of it 24~km to the east, connected by highways.}
\label{fig:instances}
\end{figure}

\subsection{Scenario Tree and Failure Modes}
\label{sec:exp_tree}

A system-wide degradation regime, which we redraw at every stage, drives the uncertainty. Every nonleaf node $m$ has one child node per regime H, N, and L (the high-, neutral-, and low-degradation regimes), with probabilities $0.35$, $0.30$, and $0.35$ that do not depend on the node. The first stage, at the root node $n_1$, observes regime N, so that $\sigma(n_1):=\mu_{\mathrm N}$, and the tree branches into the three regimes from the second stage on. Each regime scales the nominal lives by its multiplier, with $\mu_{\mathrm H}=0.66$, $\mu_{\mathrm N}=1.07$, and $\mu_{\mathrm L}=1.61$. For reference, the June--August precipitation of Sioux Falls is $1.55$ times its monthly mean~\citep{noaanormals2021}, close to $1/\mu_{\mathrm H}$. The accumulated stress of a node is the running minimum of the multipliers on its path, $\sigma(m):=\min\{\sigma(m^{-}),\mu(m)\}$, where $\mu(m)$ is the multiplier of the regime of node $m$, and the failure time of task $i$ is $\varphi_i^m:=\bar{\varphi}_i\,\sigma(m)$, where $\bar{\varphi}_i$ is its nominal life. Thus, later stages do not undo the damage of a stage in regime H. In Ext.\ 1 (Section~\ref{sec:propagation_extension}), the predecessors $\operatorname{Pred}(i)$ of task $i$ on the drainage instances are the tasks within 800~m along the pipe network that lie upstream of $i$ on an elevation grid of Sioux Falls. On the road instances, they are the uphill tasks among the six nearest tasks of $i$ in the same district. We set $\zeta_{\ell i}=4\sqrt{\text{slope}_{\ell i}}$ hours. Here $\text{slope}_{\ell i}$ is the elevation drop from $\ell$ to $i$ divided by their horizontal distance, with both measured in the same units. In Ext.\ 2 (Section~\ref{sec:decision_dependent_prob}), the vehicles intervene, the criticality $\omega_i$ increases with the failure cost, and stronger interventions shift the transition probabilities toward regimes N and L.

\paragraph{Simulator.} We evaluate the drainage policies in a continuous-time simulator (Appendix~\ref{ec:simulator}). In the base model, the scenario tree supplies failure times and fixed transition probabilities. For Ext.\ 1, we approximate the extra hazard from a failed and unserved task by shortening the failure times of the tasks that depend on it. For Ext.\ 2, clearing reduces the simulator's debris stock and changes the probabilities of the next regime. Ext.\ 2 represents this response with one probability vector per intervention level. A task fails when its accumulated hazard reaches one. We vary the propagation strength $c$, debris sensitivity $\kappa_D$, and debris memory $\varrho$. We calibrate each model from controlled runs. We then evaluate every policy in closed loop on the same 300 trajectories with common random numbers.

Table~\ref{tab:instances} lists the size of each instance. Its column Nodes gives the number of decision nodes of the scenario tree.

\begin{table}[htbp]
\centering\small
\caption{Characteristics of the instances.}
\label{tab:instances}
\begin{tabular}{llccccr}
\toprule
Instance & Network & Districts & $|\mathcal{T}|$ & $|K|$ & $T$ & Nodes \\
\midrule
R1  & Road     & 1 & 21 & 3 & 4 & 40 \\
R2  & Road     & 2 & 42 & 6 & 4 & 40 \\
R3  & Road     & 3 & 63 & 9 & 4 & 40 \\
D1  & Drainage & 1 & 40 & 3 & 6 & 364 \\
D2  & Drainage & 2 & 80 & 6 & 6 & 364 \\
D2i & Drainage & 2 & 80 & 6 & 6 & 364 \\
DS  & Drainage & 1 & 12 & 2 & 4 & 40 \\
\bottomrule
\end{tabular}
\end{table}

\subsection{Runtime of SDDiP and EXT}
\label{sec:exp_nd}

With a limit of 1,800 seconds, SDDiP certifies eight of nine road instance--model pairs and leaves a 0.38\% gap on the ninth. EXT leaves gaps of 38\% to 57\% on instance R2 and finds no feasible solution on instance R3 (Table~\ref{tab:nd-vs-ext}). Ext.\ 1 matches the base model on the road instances, while Ext.\ 2 yields lower objective values.
\begin{table}[htbp]
\centering\small\setlength{\tabcolsep}{4pt}
\caption{SDDiP versus the extended formulation (EXT), 1,800-second limit each. Times above 1,800 seconds include the time each method needs to stop after the limit. Ext.\ 1 and Ext.\ 2: models~\eqref{EXT1} and~\eqref{EXT2} of Sections~\ref{sec:propagation_extension} and~\ref{sec:decision_dependent_prob}. Gap = (UB$-$LB)/UB. Bold: certified optimal. --: no feasible solution found.}\label{tab:nd-vs-ext}
\begin{tabular}{llrrrrrrrr}\toprule
& & \multicolumn{4}{c}{SDDiP} & \multicolumn{4}{c}{EXT} \\ \cmidrule(lr){3-6}\cmidrule(lr){7-10}
Instance & Model & LB & UB & Gap (\%) & Time (s) & LB & UB & Gap (\%) & Time (s) \\ \midrule
R1 & Base & 38.09 & \textbf{38.09} & 0.00 & 57 & 38.09 & \textbf{38.09} & 0.00 & 142 \\
  & Ext.\ 1 & 38.09 & \textbf{38.09} & 0.00 & 30 & 38.09 & \textbf{38.09} & 0.00 & 95 \\
  & Ext.\ 2 & 36.97 & \textbf{36.97} & 0.00 & 265 & 36.95 & 37.10 & 0.40 & 450 \\
\midrule
R2 & Base & 76.72 & \textbf{76.72} & 0.00 & 59 & 60.40 & 102.02 & 40.80 & 1,812 \\
  & Ext.\ 1 & 76.72 & \textbf{76.72} & 0.00 & 159 & 60.00 & 97.05 & 38.18 & 1,835 \\
  & Ext.\ 2 & 76.31 & \textbf{76.31} & 0.00 & 841 & 50.72 & 117.52 & 56.84 & 1,805 \\
\midrule
R3 & Base & 109.46 & \textbf{109.46} & 0.00 & 233 & 85.87 & -- & -- & 1,813 \\
  & Ext.\ 1 & 109.46 & \textbf{109.46} & 0.00 & 340 & 86.77 & -- & -- & 1,814 \\
  & Ext.\ 2 & 107.49 & 107.90 & 0.38 & 1,835 & 66.31 & -- & -- & 1,821 \\
\bottomrule
\end{tabular}
\end{table}

\subsection{
Vehicle Depot Repositioning}
\label{sec:exp_repo}

In the previous experiments, every vehicle is home-based, that is, it ends each stage at the depot of its district. This rule reflects vehicles that are based at fixed district depots, but it can leave a district short of vehicles when its workload is heavier than that of the other district. Letting a vehicle end a stage at another depot moves capacity to where the workload is, at the price of a relocation charge.

We test vehicle repositioning on instance R2 with an imbalanced workload of 24 tasks in district~1 and 18 in district~2 (Table~\ref{tab:repositioning}). A flexible vehicle may end a stage at the other district's depot, and we charge $\delta^{\mathrm{rel}}$ for each vehicle and stage that ends away from its home depot. These vehicles travel twice as fast as the other vehicles, including in the home-based comparison. We run SDDiP for 1,800 seconds for each instance. At zero relocation charge, the flexible policies cost 9.6\% less than the home-based benchmark with one flexible vehicle and 6.8\% less with two. In each selected policy, a vehicle ends one stage away from its home depot in expectation.

\begin{table}[htbp]
\centering\small
\caption{
Vehicle repositioning on instance R2 with an imbalanced workload. We report the best policy costs available from runs of 1,800 seconds, including cross-evaluation of the flexible policies at both relocation charges. The relative gaps of the flexible policies are at most 0.9\%. We charge $\delta^{\mathrm{rel}}$ per flexible vehicle and stage that ends away from its home depot, and the last column reports the expected number of these stages.}\label{tab:repositioning}
\begin{tabular}{ccccc}\toprule
Flexible vehicles & Charge $\delta^{\mathrm{rel}}$ & Home-based cost & Flexible cost & Expected stages away \\ \midrule
1 & 0.00 & 154.14 & 139.34 & 1.00 \\
1 & 0.50 & 154.14 & 139.84  & 1.00 \\
2 & 0.00 & 139.31 & 129.86 & 1.00 \\
2 & 0.50 & 139.31 & 130.36  & 1.00 \\
\bottomrule\end{tabular}\end{table}

\subsection{
Model Evaluation}
\label{sec:Solutions}

As in Section~\ref{sec:exp_nd}, we run SDDiP for 1,800 seconds on each model and calibration setting. On instance DS, Figure~\ref{fig:sim_ext} reports, one mechanism at a time, the mean over samples 1 to 5 of the per-sample percentage difference from the base model. Table~\ref{tab:simulator} reports mean costs over samples 1 to 3, including the settings that combine propagation and debris, which Figure~\ref{fig:sim_ext} does not include. Because the two differ in samples and in averaging, they report different values for the same setting. With debris disabled, Ext.\ 1 reduces mean cost by 16.5\%, 23.0\%, 27.9\%, and 25.7\% at propagation strengths $c=8$, $16$, $32$, and $64$, respectively. With debris and no memory ($\varrho=0$), Ext.\ 2 reduces the cost by 1.3\% on average at $\kappa_D=4.5$. The reduction is 1.9\% when the criticality $\omega_i$ does not depend on the failure cost. We also test debris memory ($\varrho=1$) at $\kappa_D=4.5$ and a stormy season with a high-severity weather input probability of $0.5$.

When we calibrate Ext.\ 1 to strong propagation ($c=32$) but evaluate it without propagation, its cost increases by 16.8\% on average relative to the base model, whereas Ext.\ 2 calibrated to $\kappa_D=6$ in a setting with $\kappa_D=1.5$ changes it by at most 1.5\% (Figure~\ref{fig:sim_ext}(c)). When propagation is moderate ($c=8$), Ext.\ 1 calibrated to $c=32$ still reduces the cost by 11.2\% on average, against 16.5\% when calibrated to the true setting.

Figure~\ref{fig:failure_map} shows how Ext.\ 1 reduces the cost. Under strong propagation, the failures of the base model's policy cascade through the southern cluster, and four of its six tasks fail in 21\% to 32\% of the trajectories. The policy of Ext.\ 1 protects these tasks, whose failure frequency drops to at most 2\%, and accepts failures at two tasks from which propagation barely spreads. In Table~\ref{tab:simulator}, Ext.\ 1 reduces the cost by $3.9$ to $6.2$ in the settings with propagation. When debris and propagation act together, Ext.\ 2 reduces it by $1.88$ at $c=16$ and by $1.77$ at $c=32$.

\begin{figure}[htbp!]
\centering
\begin{subfigure}[t]{0.42\textwidth}\centering
\includegraphics{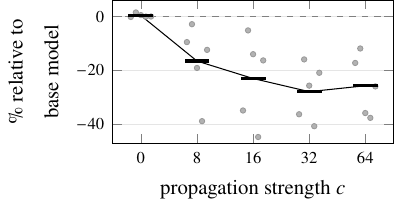}
\caption{Ext.\ 1}\end{subfigure}\hfill
\begin{subfigure}[t]{0.56\textwidth}\centering
\includegraphics{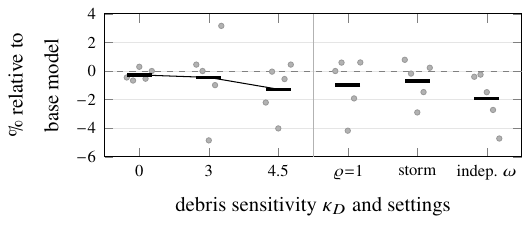}
\caption{Ext.\ 2}\end{subfigure}
\vspace{1mm}
\begin{subfigure}[t]{0.7\textwidth}\centering
\includegraphics{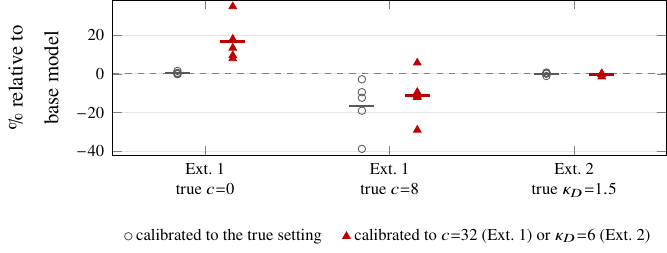}
\caption{Effect of miscalibration}\end{subfigure}
\caption{Extensions in the simulator (instance DS): relative difference (\%) between the realized cost of the extension's policy and that of the base model's policy, for five samples of instance DS (gray dots, 300 trajectories each) and their mean (black bar). Negative values favor the extension. (a) Ext.\ 1, without the debris mechanism. (b) Ext.\ 2. Section~\ref{sec:Solutions} defines the settings to the right of the line. (c) Each extension calibrated to the true setting or to a much stronger mechanism.}
\label{fig:sim_ext}
\end{figure}

\begin{figure}[htbp!]
\centering
\begin{subfigure}[t]{0.33\textwidth}\centering\includegraphics{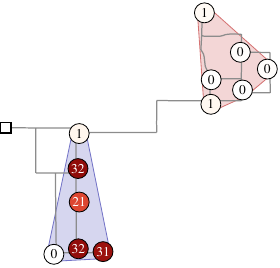}\caption{Base model, cost 12.9}\end{subfigure}\hfill
\begin{subfigure}[t]{0.33\textwidth}\centering\includegraphics{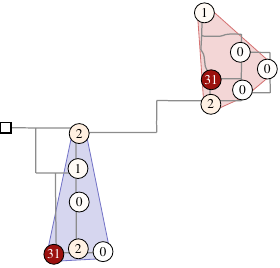}\caption{Ext.\ 1, cost 8.1}\end{subfigure}\hfill
\begin{subfigure}[t]{0.33\textwidth}\centering\includegraphics{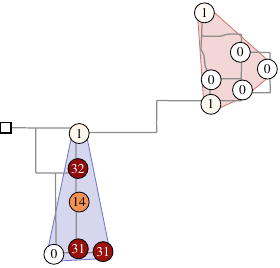}\caption{Ext.\ 2, cost 12.1}\end{subfigure}
\caption{Failure frequency (\%, darker: more frequent) of each task of instance DS under strong propagation ($c=32$) and debris, for the SDDiP policy of each model over 300 trajectories (sample 1 of instance DS, the sample with the fewest failures, whereas Table~\ref{tab:simulator} reports the mean over samples 1 to 3). Shaded: the two propagation clusters. Gray: streets. Square: depot.}
\label{fig:failure_map}
\end{figure}

\subsection{Comparison of SDDiP and HAC-RL}
\label{sec:exp_rl}

\begin{table}[htbp!]
\centering\small
\caption{Comparison of models and methods in the simulator (instance DS). We vary the intensity of propagation $c$ and the debris mechanism (none, or debris with $\kappa_D=4.5$ and $\varrho=0$, so that the debris stock resets to $D_0=0.5$ every stage). We average costs over samples 1 to 3 with 300 holdout trajectories each. We train HAC-RL on the calibrated base model for 300 seconds, or for the longest SDDiP solve of the setting when that is longer.}
\label{tab:simulator}
\begin{tabular}{lcccc}
\toprule
& \multicolumn{3}{c}{SDDiP cost} & \\
\cmidrule(lr){2-4}
Setting & Base & Ext.\ 1 & Ext.\ 2 & HAC-RL policy cost\\
\midrule
$c=0$, none & \textbf{13.96} & 14.05 & 14.02 & 15.53 \\
$c=16$, none & 22.24 & \textbf{18.37} & 21.91 & 22.07 \\
$c=32$, none & 24.42 & \textbf{18.42} & 24.40 & 24.11 \\
$c=0$, debris & 13.77 & 13.77 & \textbf{13.44} & 15.25 \\
$c=16$, debris & 22.07 & \textbf{18.09} & 20.19 & 21.67 \\
$c=32$, debris & 23.85 & \textbf{17.65} & 22.08 & 22.29 \\
\bottomrule
\end{tabular}
\end{table}

In the simulator (Table~\ref{tab:simulator}), HAC-RL trained on the calibrated base model costs more than the SDDiP policy of the same model when there is no propagation. With propagation, HAC-RL improves on the base model's policy by $0.17$ to $1.56$. Because neither policy models propagation, the base model's policy, which is optimal for its own model, is not optimal in the simulator. The HAC-RL policy, although suboptimal for the base model, adapts better to the unseen propagation. Once SDDiP models propagation through Ext.\ 1, its policy again outperforms the HAC-RL policy in every setting.

On instances R1 and R2, untrained HAC-RL policies reach dead ends despite the feasibility masks. After 900 seconds of training with the certificate, all reported policies are feasible and their mean gaps to the SDDiP optima range from 17.6\% to 25.5\% (Table~\ref{tab:rl}). SDDiP certifies these optima in 30 to 841 seconds (Table~\ref{tab:nd-vs-ext}).

\begin{table}[htbp!]
\centering\small
\caption{
HAC-RL performance after 900 seconds of training with the certificate, over five seeds. For each seed, the gap is (cost $-$ optimum)/optimum, where cost is the exact expected cost of the trained greedy policy on the scenario tree and the optimum is the SDDiP value in Table~\ref{tab:nd-vs-ext}. We report the mean gap (\%), with the minimum and maximum in parentheses, and the mean number of completed training iterations. No reported policy reaches a dead end.}
\label{tab:rl}
\begin{tabular}{llccr}
\toprule
Instance & Model & Optimum & Mean gap (\%) & Mean iterations \\
\midrule
R1 & Base & 38.09 & 20.6 (14.5--24.6) & 2,548 \\
& Ext.\ 1 & 38.09 & 17.6 (15.6--19.6) & 2,702 \\
& Ext.\ 2 & 36.97 & 21.5 (18.2--25.9) & 2,706 \\
\midrule
R2 & Base & 76.72 & 25.5 (24.1--26.7) & 842 \\
& Ext.\ 1 & 76.72 & 23.0 (22.1--23.5) & 794 \\
& Ext.\ 2 & 76.31 & 22.2 (19.5--23.4) & 788 \\
\bottomrule
\end{tabular}
\end{table}

\begin{figure}[htbp!]
\centering
\begin{minipage}{0.78\textwidth}\centering
\includegraphics{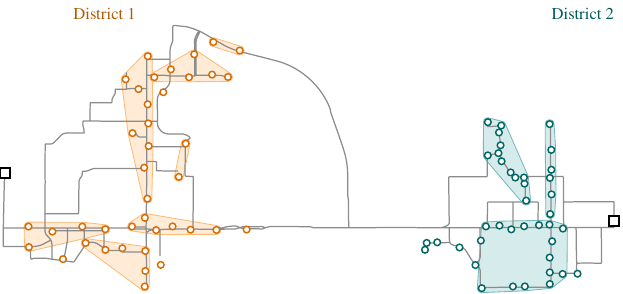}
\end{minipage}

\vspace{1mm}
\includegraphics{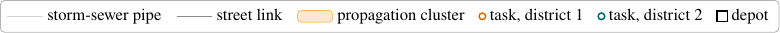}
\caption{Instance D2, the box of Figure~\ref{fig:instances}(a): 40 tasks per district, a depot outside each district, and the clusters of tasks connected by propagation (shaded). Gray: the streets of the shortest paths connecting tasks.}
\label{fig:drainage}
\end{figure}

On instances D1, D2, and D2i (Figure~\ref{fig:drainage}), neither method converges within 1,800 seconds. Therefore, Table~\ref{tab:large} compares the policies of both methods under a short budget of 600 seconds and a long budget of 7,200 seconds. After 600 seconds, HAC-RL with the certificate produces feasible policies on every instance and outperforms SDDiP on instances D1 and D2. On instance D2i, SDDiP with the certificate finds the cheaper policy, with cost $11.77$. Without the certificate, the HAC-RL policies reach dead ends on instances D2 and D2i. After 7,200 seconds, SDDiP finds cheaper policies than the HAC-RL policy on instances D1 and D2 and, with the certificate, on instance D2i. Without it, SDDiP does not improve on its initial feasible solution on instance D2i. Without the certificate, the HAC-RL policy learns to avoid dead ends on instance D2i, where it reaches $10.57$, but not on instance D2. SDDiP also provides a lower bound that measures the quality of any policy. After 7,200 seconds, the gap (cost$-$LB)/cost of its policies, defined as in Table~\ref{tab:nd-vs-ext}, is 2.4\% on instance D1, 32\% on instance D2i, and 49\% on instance D2. Therefore, HAC-RL provides good policies earlier on instances D1 and D2, whereas SDDiP provides better policies and a bound on their quality with more time. Appendix~\ref{ec:results} illustrates the routes on instances R2 and D2 and compares optimal solutions of the extended formulation with SDDiP policies.

\begin{table}[htbp!]
\centering\small\setlength{\tabcolsep}{6pt}
\caption{Instances D1, D2, and D2i: cost of the policy of each method after 600 and 7,200 seconds, alongside the lower bound (LB). All runs use the base model. HAC-RL costs at 600 seconds are means over three runs, and entries at 7,200 seconds are single runs. + cert.: with the certificate of Section~\ref{sec:feasibility}. Every SDDiP run starts from a feasible solution found by a sequential local search heuristic. --: no feasible policy. n/a: not applicable.}
\label{tab:large}
\begin{tabular}{llrrrr}
\toprule
 & & 600 s & \multicolumn{3}{c}{7,200 s} \\
\cmidrule(lr){3-3}\cmidrule(lr){4-6}
Instance & Method & Cost & Cost & LB & Iterations \\
\midrule
D1 
 & SDDiP         & 78.08 & 2.88  & 2.81 & 43 \\
 & HAC-RL            & 10.04 & 4.39  & n/a & 7,180 \\
 & HAC-RL + cert.    &  5.64 & 4.61  & n/a & 5,810 \\
\midrule
D2 
 & SDDiP         & 246.39 & 8.62  & 4.42 & 21 \\
 & HAC-RL            & --     & --    & n/a & 1,630 \\
 & HAC-RL + cert.    & 21.21  & 16.00 & n/a & 1,110 \\
\midrule
D2i 
 & SDDiP         & 246.39 & 246.39 & 4.50 & 25 \\
 & SDDiP + cert. & 11.77  & 6.63   & 4.50 & 24 \\
 & HAC-RL            & --     & 10.57  & n/a & 2,590 \\
 & HAC-RL + cert.    & 64.19  & 13.62  & n/a & 1,780 \\
\bottomrule
\end{tabular}
\end{table}

\section{Conclusion}
\label{sec:conclu}

We developed a multi-stage model for infrastructure inspection that couples the routing of each stage with failure risk. Its extensions capture propagation from failed, unserved tasks and intervention-dependent transition probabilities. A route-based stage decomposition supports an SDDiP-type algorithm with finite convergence under the stated assumptions. On the road instances, SDDiP proves optimality where the extended formulation leaves large gaps within the same budget. Its cuts also define the SDDiP policy. The hierarchical actor-critic learns assignments from the observed state and uses an immediate-cost routing oracle. With the certificate, HAC-RL finds cheaper policies than SDDiP within 600 seconds on two of the three six-stage instances. Within 7,200 seconds, SDDiP finds cheaper policies on all three (on instance D2i, with the certificate) and provides lower bounds on the optimal costs. In the simulator, the results also show the value of modeling propagation and decision-dependent transitions when these mechanisms are present.

We could reduce the number of cost-to-go approximations by sharing cuts among nodes with the same stage and accumulated stress. For longer horizons or larger task budgets, column generation could replace route enumeration. Extending the model to failures that block access would require state-dependent routing networks, while failures that only a visit reveals would require partial observability.

\bibliographystyle{informs2014trsc}

\ECSwitch
\renewcommand*{\theHsection}{EC.\arabic{section}}
\renewcommand*{\theHsubsection}{\theHsection.\arabic{subsection}}
\ECHead{E-Companion}

\section{Proofs of Propositions~\ref{prop:node_exact}, \ref{prop:inert}, and~\ref{prop:valid}, and Theorem~\ref{thm:snd_convergence}}
\label{ec:proofs}

\proof{Proof of Proposition~\ref{prop:node_exact}.}
Fix a node $n$ and an admissible inherited state. The binary assignment and depot-flow constraints of the base model select, for every vehicle, a nonempty depot-to-depot path through its assigned tasks. Every resulting path is elementary because a task has at most one selected incoming and outgoing arc, and no disconnected cycle of tasks exists, by the timing argument following routing block~\eqref{OP}. We can replace a feasible schedule by its earliest schedule without changing depots, assignments, or travel cost and without turning a pre-failure inspection into a post-failure one. Conversely, every nonempty route column and its earliest schedule satisfy the routing and assignment blocks of the base model. Under the nonnegative failure costs and the specified transition rules, advancing service and setting the indicators for all admissible on-time completions cannot increase present or future costs. Therefore, there is an optimal policy of the base model represented by nonempty columns, with the same state updates and conditional expected cost. Appending dummy columns and terminal slacks makes the Bellman recursion finite and yields $Q_n=V^n$ for every admissible inherited state.

Now suppose the unpenalized subtree rooted at $n$ has a feasible continuation. Its expected cost is at most $C_{\mathrm{feas}}$, because this number bounds every complete-path cost. Consider any penalized policy that uses a dummy route or violates terminal coverage anywhere in the subtree. Such a violation occurs on at least one complete descendant path. Its conditional probability is the product of the transition probabilities along the path from $n$ to its leaf. Each conditional factor is at least the factor of its edge in definition~\eqref{eq:pmin}, which under Ext.\ 2 in~\eqref{EXT2}, is the smallest probability of the edge across intervention levels. Adding the missing prefix factors, all at most one, can only reduce the product. Thus $p_{\min}$ also bounds the probability of the descendant path conditional on reaching $n$. The violation incurs at least one penalty $P$, and all other primary costs are nonnegative. Hence the conditional expected penalized cost is at least $p_{\min}P>C_{\mathrm{feas}}$. It is strictly larger than the cost of a feasible, penalty-free continuation. Therefore no optimal penalized policy at a feasible inherited state incurs a penalty, on any descendant path. The route correspondence established above proves $V^n=V_{\mathrm{base}}^n$ on those states. The equivalence follows by applying this equality at the feasible root. \Halmos
\endproof

\proof{Proof of Proposition~\ref{prop:inert}.}
Because constraint~\eqref{EXT:once} prevents any new assignment of a completed task, $q_i$ enters no descendant constraint other than its own persistence, and its failure increment is zero. Under Ext.\ 1 in~\eqref{EXT1}, $O_i$ enters only the failure times of the tasks $j$ with $i\in\operatorname{Pred}(j)$, all of which the vehicles have already served and hence protected from further failure. Two states that agree on the remaining coordinates therefore have the same feasible sets and costs in every descendant node, and thus they have the same value. If a binary state $\mathbf{s}$ agrees with $\hat{\mathbf{s}}$ on $J(\hat{\mathbf{s}})$, the right-hand side of cut~\eqref{eq:ioc_cut} is $\underline{Q}_m(\hat{\mathbf{s}})\le V^m(\hat{\mathbf{s}})=V^m(\mathbf{s})$, because we compute $\underline{Q}_m$ with under-estimators of the cost-to-go functions of the child nodes. Otherwise, the sum in cut~\eqref{eq:ioc_cut} is at least one and $\underline{Q}_m(\hat{\mathbf{s}})\ge0$. Thus, the right-hand side is at most $0\le V^m(\mathbf{s})$. \Halmos
\endproof

\proof{Proof of Theorem~\ref{thm:snd_convergence}.}
We give the argument for the iterations that traverse the whole tree. The sampled iterations between them only add valid cuts, which can only increase the approximations and do not affect the argument. Let $C_n(x)$ be the immediate cost of decision $x$ in objective~\eqref{NODE:obj} and $\mathbf{s}_m(x)$ the state that $x$ passes to child node $m$. Write $p_{nm}(x)$ for the transition probability under decision $x$, which is constant in the base model~\eqref{EXT} and Ext.\ 1 in~\eqref{EXT1}, and follows the selected intervention level in Ext.\ 2 in~\eqref{EXT2}. Throughout the proof, $V^m(\mathbf{s})$ is the finite cost-to-go of the penalized recursion defined in Section~\ref{sec:two_horizon}. We call a pair $(m,\mathbf{s})$ of a node $m$ with $t(m)\ge2$ and a binary state exact if $\underline{V}^m(\mathbf{s})=V^m(\mathbf{s})$. Because all cuts are valid and we only add cuts, $\underline{V}^m\le V^m$, $\mathrm{LB}$ never exceeds the optimal value $\mathrm{OPT}$, and an exact pair remains exact. Under a finite scenario tree and a pure-binary state, there are finitely many pairs. Consider the forward pass of an iteration, in which node $n$ receives state $\hat{\mathbf{s}}_n$, chooses $\hat{x}_n$, and passes $\hat{\mathbf{s}}_m=\mathbf{s}_m(\hat{x}_n)$ to each child node $m\in\mathcal{C}(n)$.

(a) If every pair $(m,\hat{\mathbf{s}}_m)$ is exact, then for every node $n$ and every feasible decision $x$,
\begin{align*}
C_n(\hat{x}_n)+\sum_{m\in\mathcal{C}(n)}p_{nm}(\hat{x}_n)V^m(\hat{\mathbf{s}}_m) &= C_n(\hat{x}_n)+\sum_{m\in\mathcal{C}(n)}p_{nm}(\hat{x}_n)\underline{V}^m(\hat{\mathbf{s}}_m)\\
&\le C_n(x)+\sum_{m\in\mathcal{C}(n)}p_{nm}(x)\underline{V}^m(\mathbf{s}_m(x))\\
&\le C_n(x)+\sum_{m\in\mathcal{C}(n)}p_{nm}(x)V^m(\mathbf{s}_m(x)),
\end{align*}
and therefore $\hat{x}_n$ is optimal at $\hat{\mathbf{s}}_n$ and the value of node $n$ equals $V^n(\hat{\mathbf{s}}_n)$. Hence $\mathrm{LB}=\mathrm{OPT}$, the SDDiP policy is optimal, $\mathrm{UB}=\mathrm{LB}$, and the method stops.

(b) Otherwise, let $(n,\hat{\mathbf{s}}_n)$ be a pair that is not exact with $t(n)$ maximal. For every child node $m$ of $n$, the pair $(m,\hat{\mathbf{s}}_m)$ is exact. If $n$ is a leaf, it has no child nodes and its node problem is exact. As in (a), the forward-pass value of node $n$ is $C_n(\hat{x}_n)+\sum_{m\in\mathcal{C}(n)}p_{nm}(\hat{x}_n)V^m(\hat{\mathbf{s}}_m)\ge V^n(\hat{\mathbf{s}}_n)$. Because $\underline{V}^m\le V^m$, this value is also at most $V^n(\hat{\mathbf{s}}_n)$. Hence, it equals $V^n(\hat{\mathbf{s}}_n)$. The backward pass solves the node problem at $\hat{\mathbf{s}}_n$ with child-node approximations that are pointwise no smaller than those of the forward pass and remain bounded above by the true values. Because the forward-pass optimum already equals $V^n(\hat{\mathbf{s}}_n)$, the backward-pass optimum equals it as well. The integer optimality cut, which Proposition~\ref{prop:inert} establishes as valid, therefore makes $(n,\hat{\mathbf{s}}_n)$ exact.

Every complete-tree traversal therefore either certifies $\mathrm{LB}=\mathrm{UB}=\mathrm{OPT}$ or makes a previously inexact node--state pair exact. Because the number of such pairs is finite, the number of nonterminating complete traversals is finite. With a complete traversal at least once in every $\nu$ iterations, the total number of iterations is finite as well. By Proposition~\ref{prop:node_exact}, $\mathrm{OPT}$ is the optimal value of model~\eqref{EXT} when that formulation admits a feasible policy. \Halmos
\endproof

\proof{Proof of Proposition~\ref{prop:valid}.}
(i) Consider a feasible completion from the state at the end of node $n$. In each remaining stage, every vehicle of $K_g$ performs a nonempty route, by constraint~\eqref{EXT:active_vehicle}, whose tasks belong to $\mathcal{T}_g$ and are open, by constraint~\eqref{EXT:once}, and it serves at most $B_{k\tau}$ tasks. By constraint~\eqref{EXT:terminal_coverage}, the vehicles serve every open task of $\mathcal{T}_g$ by the end of the horizon, and only vehicles of $K_g$ may serve these tasks. Hence, the number of open tasks of $g$ satisfies inequalities~\eqref{eq:reach}. Moreover, the task sets of the $\rho^n_g$ routes of $K_g$ are pairwise disjoint, belong to $\mathcal{G}_g$, and cover the open tasks of $g$. Thus, they partition the open tasks. A state that violates inequalities~\eqref{eq:reach} or that the certificate rejects therefore is a dead end. Conversely, if the vehicles of $g$ are identical, home-based at a common depot, and have the same routes in every stage, we can assign the $\rho^n_g$ sets of a partition to these routes in any order, with $|K_g|$ sets per remaining stage, and every set is the task set of a route that each vehicle of $K_g$ can perform in that stage.

(ii) Let $V^m$ denote the value of the subtree of $m$ with the penalties of node problem~\eqref{NODE}, as in the proof of Theorem~\ref{thm:snd_convergence}. Fix a completion and a root-to-leaf path of the subtree. On this path, vehicle $k$ performs in stage $\tau$ a route $\varsigma_{k\tau}\in\mathcal{R}_{k\tau}$, possibly empty at cost $P$, and $w_i\in\{0,1\}$ marks a task left without completion at cost $P$, so that $\mathrm{cov}_i+w_i\ge1-v_i^n$, where $\mathrm{cov}_i$ is the number of routes that serve task $i$. Failure costs, relocation costs, and penalized slacks are nonnegative. The secondary tie-breaking criterion does not alter the primary objective. For $\boldsymbol\beta\in[0,P]^{|\mathcal{T}|}$, the cost of the path is therefore at least
\begin{align*}
    \sum_{\tau=t(n)+1}^{T}\sum_{k\in K}\delta_{\varsigma_{k\tau}}+P\sum_{i\in\mathcal{T}}w_i & \ge\ \sum_{\tau=t(n)+1}^{T}\sum_{k\in K}\Bigl(\delta_{\varsigma_{k\tau}}-\sum_{i\in S_{\varsigma_{k\tau}}}\beta_i\Bigr)+\sum_{i\in\mathcal{T}}\beta_i(\mathrm{cov}_i+w_i) \\
    & \ge\ \sum_{\tau=t(n)+1}^{T}\sum_{k\in K}\Bigl(\delta_{\varsigma_{k\tau}}-\sum_{i\in S_{\varsigma_{k\tau}}}\beta_i\Bigr)+\sum_{i\in\mathcal{T}}\beta_i(1-v_i^n),
\end{align*}
because $Pw_i\ge \beta_iw_i$. Each term of the first sum is at least $\min\{0,\cdot\}$ of the inner minimum of cut~\eqref{eq:coverage_cut}, because an empty route costs $P\ge0$. Because the bound does not depend on the path, it holds for the expected cost under any transition probabilities, including those of Ext.\ 2 in~\eqref{EXT2}, and for the optimal completion.

(iii) Inequalities~\eqref{eq:reach} and~\eqref{eq:nogood} charge $P$ only to outgoing states that are dead ends, by (i). At a state admitting a feasible continuation in the base model, Proposition~\ref{prop:node_exact} and condition~\eqref{eq:penalty_sufficient} show that every optimal penalized continuation is penalty-free on all descendant paths. Consequently, penalizing a rejected outgoing state cannot improve the optimum or eliminate all its optimal solutions, and thus the equality with the optimal value of the base model still holds at feasible states. Inequality~\eqref{eq:coverage_cut} is valid by (ii). We add no-good inequalities only finitely often, because there are finitely many coverage vectors. Each addition can only increase the cost-to-go functions, and thus the cuts generated before it remain valid. The proof of Theorem~\ref{thm:snd_convergence} applies to the iterations after the last addition. \Halmos
\endproof

\section{Hierarchical Actor-Critic: Algorithm and Feasibility}
\label{ec:rl}

\begin{algorithm}[htbp]
\small
\caption{Hierarchical actor-critic learning with embedded routing evaluation}
\label{alg:rl}
\begin{algorithmic}[1]
\State \textbf{Inputs:} horizon $T$, simulator implementing endogenous failure dynamics, routing evaluator, optional certificate.
\State Initialize actor parameters $\Theta$ and critic parameters $\psi$.
\For{training iteration $=1,2,\dots$}
\For{each of 16 sampled root-to-leaf paths}
\State Initialize state $\mathbf{s}_0=(\mathbf{v}_0,\mathbf{q}_0,\mathbf{e}_0,\boldsymbol{\Phi}_0)$.
\For{$t=1,\dots,T$}
\State Build $\mathbf{x}_t$ one task--vehicle pair at a time, sampling from $\pi_{\Theta}(\cdot \mid \mathbf{s}_{t-1})$ under masks updated after every choice (and resample if the certificate rejects $\mathbf{x}_t$).
\State Evaluate the routing layer to obtain $R_t(\mathbf{x}_t;\mathbf{e}_{t-1})$ and compute the next state $\mathbf{s}_t$.
\State Compute stage cost $C_t$ via \eqref{eq:rl_stage_cost} and store $(\mathbf{s}_{t-1},\mathbf{x}_t,C_t,\mathbf{s}_t)$.
\EndFor
\EndFor
\State Compute Monte Carlo returns and advantages for the batch of 16 trajectories.
\State Update critic parameters $\psi$ by minimizing the squared advantages over the batch.
\State Update actor parameters $\Theta$ using the batch estimate of the policy gradient~\eqref{eq:policy_gradient}, with the proposal-score approximation when we use certificate resampling.
\EndFor
\end{algorithmic}
\end{algorithm}

Algorithm~\ref{alg:rl} summarizes the training of the hierarchical actor-critic of Section~\ref{subsec:rl_actor_critic}. With the certificate, we reject and resample assignments that the certificate of Section~\ref{sec:feasibility} rejects. We allow up to 30 attempts. If we reject every attempt, we execute the last assignment and charge the dead end a penalty of $P$. The actor and the critic each have one hidden layer of 16 tanh units. The actor scores 16 features of each task--vehicle pair. The critic uses seven features of the stage: (i) the stage fraction $t/T$, (ii) the fraction of open tasks, (iii) the fraction of tasks whose failure risk exceeds a threshold and that have not yet failed, (iv) the fraction of failed tasks, (v) the fraction of tasks that are propagation sources under Ext.\ 1 in~\eqref{EXT1}, (vi) the workload pressure, that is, the number of open tasks divided by the product of $|K|$, the task budget, and the number of remaining stages, and (vii) a constant. Each training iteration samples 16 root-to-leaf paths. Each update uses Adam with learning rate $5\times10^{-3}$, gradient clipping, and an entropy bonus of weight $0.01$. For certificate-assisted training, we use the accepted construction's log-probability under the actor's policy $\pi_\Theta$ in the actor update. The update omits the certificate's acceptance-probability correction and therefore uses a proposal-score approximation to the policy gradient of the resampling procedure. We report the greedy policy, which selects the candidate with the highest probability at every step of the construction, and we evaluate it on the tree every 10 updates.

\section{Simulator Details}
\label{ec:simulator}

\paragraph{Dynamics.} Each unserved task accumulates a hazard $\Lambda_i$ at rate $(1+\Gamma_i)/(\bar{\varphi}_i\,\mu\,\varepsilon_i)$, where $\mu\in\{\mu_{\mathrm H},\mu_{\mathrm N},\mu_{\mathrm L}\}$ is the multiplier of the regime of the stage, $\varepsilon_i\in[0.05,1]$ is a local factor that decreases with the exposure of task $i$ to moving storm cells, and $\Gamma_i$ is the propagation term. The task fails when $\Lambda_i\ge1$, and its hazard stops growing once a vehicle serves it. Each failed and unserved task $j$ upstream of $i$ adds to $\Gamma_i$ a term proportional to $c$. The debris stock $D\in[0,1]$ increases with rain, decreases in proportion to the intervention intensity with coefficient $0.30$, and reverts toward its initial level $D_0=0.5$ at rate $1-\varrho$. The next regime is $\mathrm H$, $\mathrm N$, or $\mathrm L$ according to thresholds on a latent severity, which is the sum of a weather draw, the term $\kappa_D(D-\tfrac12)$, and Gaussian noise with standard deviation $0.15$. At the start of each stage, every policy observes the failure status of every task.

\paragraph{Calibration and policies.} 
We simulate at least 32 training trajectories per leaf, with seeds disjoint from the evaluation seeds. For each task $i$ and tree node $n$, we average the cumulative hazard at the end of the stage over trajectories whose regime paths match $n$. We set the failure time to the first time this conditional mean hazard reaches one, using linear interpolation within the stage. Once we assign a failure time, every descendant node inherits it. We estimate the propagation coefficient of Ext.\ 1 in~\eqref{EXT1}, from paired runs with and without a failed upstream source, then scale it by the fraction of model triggers that the simulator realizes as unserved failures. For the base model~\eqref{EXT} and Ext.\ 1 in~\eqref{EXT1}, we use the regime frequencies under the status-quo intervention level, which is half the intervention capacity. For Ext.\ 2 in~\eqref{EXT2}, we fit the intervention threshold $u_1$ by least squares. The first intervention level uses the average regime law below $u_1$, and the second uses the regime law at $u_1$. We keep evaluation seeds disjoint from calibration, training, and validation seeds.

\section{Additional Numerical Results}\label{ec:results}

\subsection{Solution Visualization}
Figure~\ref{fig:routes} shows the routes of the SDDiP policies on instances R2 and D2, in the first stage and in the second stage in regime H. On instance R2, a route spans a city and decides which tasks the vehicle serves before they fail. On instance D2, the routes are short and the number of tasks per stage binds.

\begin{figure}[htbp!]
\centering
\begin{subfigure}[t]{0.49\textwidth}\centering
\includegraphics{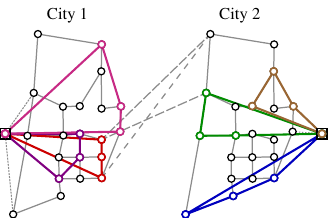}
\caption{Instance R2, first stage}\end{subfigure}\hfill
\begin{subfigure}[t]{0.49\textwidth}\centering
\includegraphics{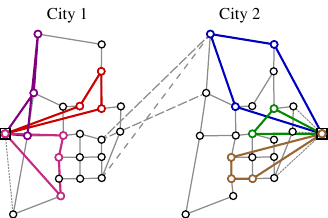}
\caption{Instance R2, second stage in regime H}\end{subfigure}

\vspace{1mm}
\begin{subfigure}[t]{0.49\textwidth}\centering
\includegraphics{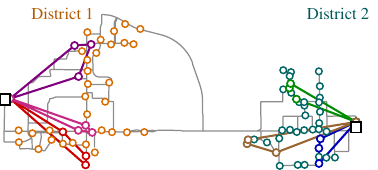}
\caption{Instance D2, first stage}\end{subfigure}\hfill
\begin{subfigure}[t]{0.49\textwidth}\centering
\includegraphics{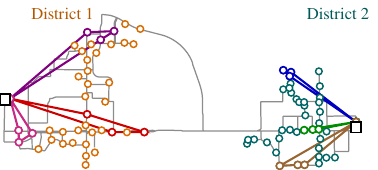}
\caption{Instance D2, second stage in regime H}\end{subfigure}
\caption{Routes of the six vehicles 
(one color per vehicle) in the SDDiP policy. We color the markers for tasks served. The first stage observes regime N, while the second stage observes regime H. We plot all other elements as in Figures~\ref{fig:instances} and~\ref{fig:drainage}.}
\label{fig:routes}
\end{figure}

\subsection{Extended Formulation Versus SDDiP Policy}
\label{ec:feedback}
\begin{table}[htbp!]
\centering\small
\caption{Optimal solution of the extended formulation versus SDDiP policy in the simulator on instance DS, same samples as Table~\ref{tab:simulator}. EXT: an optimal solution of the extended formulation of each model. SDDiP: the policy of Table~\ref{tab:simulator}. Bold: the lowest cost of each setting.}
\label{tab:feedback}
\begin{tabular}{lcccccc}
\toprule
& \multicolumn{2}{c}{Base} & \multicolumn{2}{c}{Ext.\ 1: model~\eqref{EXT1}} & \multicolumn{2}{c}{Ext.\ 2: model~\eqref{EXT2}} \\
\cmidrule(lr){2-3}\cmidrule(lr){4-5}\cmidrule(lr){6-7}
Setting & EXT & SDDiP & EXT & SDDiP & EXT & SDDiP \\
\midrule
$c=0$, none & \textbf{13.94} & 13.96 & 14.04 & 14.05 & 14.03 & 14.02 \\
$c=16$, none & 21.65 & 22.24 & 18.40 & \textbf{18.37} & 22.72 & 21.91 \\
$c=32$, none & 24.57 & 24.42 & \textbf{18.41} & 18.42 & 25.20 & 24.40 \\
$c=0$, debris & 13.75 & 13.77 & 13.76 & 13.77 & \textbf{13.44} & \textbf{13.44} \\
$c=16$, debris & 21.53 & 22.07 & 18.10 & \textbf{18.09} & 20.62 & 20.19 \\
$c=32$, debris & 24.27 & 23.85 & 17.77 & \textbf{17.65} & 23.45 & 22.08 \\
\bottomrule
\end{tabular}
\end{table}
Table~\ref{tab:feedback} compares an optimal solution of the extended formulation with the SDDiP policy of Table~\ref{tab:simulator}, which reacts to the observed state. For the base model~\eqref{EXT} and Ext.\ 1 in~\eqref{EXT1}, the two policies differ by at most 0.59 and 0.12, respectively. Neither dominates. For Ext.\ 2 in~\eqref{EXT2}, the SDDiP policy outperforms the solution of the extended formulation in every setting with propagation, by 0.43 to 1.37. The SDDiP policy therefore matters most when the transitions depend on the decisions.

\end{document}